\documentclass[11pt]{amsart}

\title[Topology of the links of cDV singularities]{Topology of the links of cDV singularities of types $cA_n$ for $n>0$ \\ and $cD_n$ for $n>4$}

\author[Ishikawa]{Masaharu Ishikawa}
\address{Department of Mathematics, Hiyoshi Campus, Keio University, 4-1-1 Hiyoshi, Kohoku, Yokohama, Kanagawa, 223-8521, Japan}
\email{ishikawa@keio.jp}

\author[Katanaga]{Atsuko Katanaga}
\address{
Faculty of Education, Iwate University, 3-18-8 Ueda, Morioka, Iwate, 020-8550, Japan}
\email{katanaga@iwate-u.ac.jp}

\usepackage{amsfonts,amsmath,amssymb,amscd}
\usepackage{amsthm}
\usepackage{latexsym}
\usepackage{graphicx}
\usepackage{psfrag}
\usepackage{color}

\usepackage{fancybox, ascmac}
\usepackage{multicol}
\usepackage{ulem} 
\usepackage{mathdots}

\usepackage[colorlinks=true,linkcolor=black,citecolor=black,urlcolor=black]{hyperref}

\theoremstyle{plain}
\newtheorem*{theorem*}{Theorem}
\newtheorem*{lemma*} {Lemma}
\newtheorem*{corollary*} {Corollary}
\newtheorem*{proposition*}{Proposition}
\newtheorem*{conjecture*}{Conjecture}
\newtheorem{theorem}{Theorem}[section]
\newtheorem{lemma}[theorem]{Lemma}

\theoremstyle{remark}

\newtheorem{definition}[theorem]{Definition}
\newtheorem{example}[theorem]{Example}
\newtheorem{remark}[theorem]{Remark}

\newtheorem*{example*}{Example}

\theoremstyle{definition}

\newtheoremstyle{citing}
  {}
  {}
  {\itshape}
  {}
  {\bfseries}
  {.}
  {.5em}
  {\thmnote{#3}}

\theoremstyle{citing}

\newcommand{\N}{\mathbb{N}}
\newcommand{\C}{\mathbb{C}}
\newcommand{\Z}{\mathbb{Z}}
\newcommand{\R}{\mathbb{R}}
\newcommand{\Q}{\mathbb{Q}}
\newcommand{\Sing}{\text{\rm Sing}}
\newcommand{\hd}{\hspace{0.3mm}|\hspace{0.3mm}}
\newcommand{\hnd}{\hspace{-1.7mm}{\not|}\hspace{0.3mm}}

\makeatletter
\@addtoreset{equation}{section}
\def\theequation{\thesection.\arabic{equation}}
\makeatother

\allowdisplaybreaks

\begin{document}

\begin{abstract}
We show that the second integral homology group of the link of an isolated compound Du Val (cDV, for short) singularity of type $cA_n$ is either trivial or a torsion-free abelian group. Consequently, by a result of Smale, it follows that the link is either $S^5$ or a connected sum of finitely many copies of $S^2\times S^3$. 
We also determine the rank of the second integral homology group of the link of a singularity of type $cD_n$ with $n>4$ under the assumption that the singularity is Newton non-degenerate,
using Varchenko's formula.
Furthermore, we focus on the weighted homogeneous case and determine the homology group of the link, including its torsion subgroup,
under the assumption that the singularity is a Thom-Sebastiani sum of singularities of Brieskorn-Pham, cyclic, or chain type.
\end{abstract}

\maketitle

\section{Introduction}


Let $f:\C^k\to\C$ be a polynomial in $k$ variables such that $f(O)=0$, where $O$ is the origin,
and let $(f,O)$ denote the singularity of $f$ at $O$.
For a sufficiently small $\varepsilon>0$, the intersection $K_f=S_\varepsilon^{2k-1}\cap f^{-1}(0)$,
where $S_\varepsilon^{2k-1}$ is the $(2k-1)$-dimensional sphere centered at $O$ of radius $\varepsilon>0$,
is called the {\it link} of $(f,O)$. 
It is known from~\cite{Mil68} that the pair $(D^{2k}_\varepsilon, f^{-1}(0))$, 
where $D_\varepsilon^{2k}$ is the $2k$-dimensional ball bounded by $S_\varepsilon^{2k-1}$, is homeomorphic to the cone over the pair $(S_\varepsilon^{2k-1}, K_f)$. Therefore, the topology of $(f,O)$ is determined by $(S_\varepsilon^{2k-1}, K_f)$.
Note that $K_f$ is a manifold if $(f,O)$ is an isolated singularity.
The topology of links of isolated complex hypersurface singularities is a fundamental problem in singularity theory. Although powerful tools, including Milnor fiber theory, monodromy, and surgery theory, have provided a deep understanding of these links, explicit computations of their integral homology groups remain highly nontrivial, especially for $k\geq 4$.
Only a few explicit results are known; see, for example,~\cite{KN08, Kol05, Kol09}.

Compound Du Val (cDV, for short) singularities form an important class of singularities that naturally arise in the study of algebraic threefolds and play a central role in the birational geometry of threefolds.
In this paper, we focus on cDV singularities as hypersurface singularities in four variables, 
corresponding to the case $k=4$.
They are classified into five types: $cA_n$, $cD_n$, $cE_6$, $cE_7$, and $cE_8$.
We study the topology of isolated cDV singularities of types $cA_n$ for $n>0$ and $cD_n$ for $n>4$.
An isolated cDV singularity of type $cA_n$ is given in the form
\[
   f(x,y,z,w)=x^2+y^2+h(z,w),
\]
where $h(z,w)$ is assumed to have an isolated singularity at the origin.
Let $H_n(M)$ denote the $n$-th integral homology group of a manifold $M$.
Unless otherwise specified, all homology groups in this paper are taken with coefficients in $\Z$.
It follows from the results of Flenner~\cite{Fle81} and Koll\'{a}r~\cite[Proposition~2.2.6]{Kol91} that
the second homology group $H_2(K_f)$ of the link $K_f$ of an isolated cDV singularity of type $cA_n$
has rank $m-1$, where $m$ is the number of local irreducible components of $h(z,w)=0$ at the origin.
In the following theorem, we show, using techniques from low-dimensional topology, that $H_2(K_f)$ is torsion-free.
Combined with Smale's result~\cite{Sma62}, this completely determines the diffeomorphism type of $K_f$. 

\begin{theorem}\label{thm1}
Suppose that $f(x,y,z,w)=x^2+y^2+h(z,w)$ has an isolated cDV singularity of type $cA_n$ at the origin.
Then, 
\[
H_2(K_f)\cong\Z^{m-1},
\]
where
$m$ is the number of local irreducible components of $h(z,w)=0$ at the origin.
In particular, $K_f$ is diffeomorphic to $S^5$ if $m=1$, and otherwise is diffeomorphic
to the connected sum of $m-1$ copies of $S^2\times S^3$.
\end{theorem}

The main focus of this paper is the study of Newton non-degenerate, isolated cDV singularities of type $cD_n$ with $n>4$.
An isolated cDV singularity of type $cD_n$ is given in the form
\begin{equation}\label{eq1-1}
   f(x,y,z,w)=x^2+y^2z+ayw^\kappa+h(z,w),\\
\end{equation}
where $\kappa\geq 3$, $h(z,w)$ is not identically zero, and the lowest degree of $h(z,w)$ is at least $4$,
see~\cite[Theorem 2.9]{Kol98}. 
A necessary and sufficient condition for this singularity to be isolated is given in Lemma~\ref{lemma21}.
The definition of Newton non-degeneracy is given in Section~\ref{sec21}. 
Note that Newton non-degeneracy is a generic condition.
Under these assumptions, we can determine the zeta function of the Milnor fibration of $(f,O)$ from its Newton boundary and obtain the information on $H_2(K_f)$.

To state the theorem, we fix notation for the coordinates of the vertices of the Newton boundary.
The Newton polygon $\Gamma_+(h)$ of the singularity $(h,O)$, where $h(z,w)=\sum_{i,j}a_{ij}z^iw^j$, is defined by the convex hull of 
$\{(i,j)+\R_{\geq 0}^2\mid a_{ij}\ne 0\}\subset\R^2$,
where $\R_{\geq 0}=\{t\in\R\mid t\geq 0\}$.
The Newton boundary $\Gamma(h)$ is defined by the union of the compact faces of $\Gamma_+(h)$.
Let $\Delta_i$, $i=1,\ldots, N$ be the $1$-dimensional compact faces of $\Gamma_+(h)$,
and $(u_{i-1}, v_{i-1})$ and $(u_i, v_i)$ be its two endpoints, where $v_{i-1}>v_i$. See Figure~\ref{fig5} 
in Section~\ref{section4-1}.
Set $\ell_i=\gcd(u_i-u_{i-1}, v_{i-1}-v_i)$, which is the lattice length of $\Delta_i$.

The Newton polygon $\Gamma_+(f)$ of $(f,O)$ in the form~\eqref{eq1-1}
is the convex hull of $(0,0)\times \Gamma_+(h)+\mathbb R_{\geq 0}^4$,
$(2,0,0,0)+\mathbb R_{\geq 0}^4$, and $(0,2,1,0)+\mathbb R_{\geq 0}^4$, and additionally $(0,1,0,\kappa)+\mathbb R_{\geq 0}^4$ if $a\ne 0$.
The Newton boundary $\Gamma(f)$ is the union of compact faces of $\Gamma_+(f)$.
There exists a vertex $(u_c, v_c)$ of $\Gamma(h)$ such that 
the convex hull of $\Delta_i$ and the two vertices $(2,0,0,0)$ and $(0,2,1,0)$ is a face of $\Gamma(f)$
for each $i\in\{c+1,\ldots,N\}$, see Lemma~\ref{lemma43z}.

As mentioned in~\cite[Remark~4.7]{Cai03}, if $(f,O)$ is Newton non-degenerate,
$\text{\rm rank\,}H_2(K_f)$ can be computed using Varchenko's formula~\cite{Var76} (see Theorem~\ref{thmvar} below). 
Applying this formula, we obtain the following theorem, which explicitly determines $\text{\rm rank\,}H_2(K_f)$ for a Newton non-degenerate isolated cDV singularity $(f,O)$ of type $cD_n$ with $n>4$ from its Newton boundary.
Let $\nu_2(\alpha)$ denote the $2$-adic valuation of an integer $\alpha$, that is, the integer $\beta$ such that $\alpha=2^\beta\gamma$ with $\gamma$ odd.

\begin{theorem}\label{thm2}
Let $(f,O)$ be a Newton non-degenerate, isolated cDV singularity of type $cD_n$ with $n>4$
given in the form~\eqref{eq1-1}.
Then,
\[
\text{\rm rank\,}H_2(K_f)=\sum_{i} \ell_i+ \varepsilon_1+\varepsilon_2,
\]
where
\begin{itemize}
\item 
the summation runs over all indices $i$ with $c+1\leq i\leq N$ satisfying $\nu_2(v_{i-1}-v_i)>\nu_2(\ell_i)$,
\item $\varepsilon_1=\gcd(2\kappa-{v_{c}}, u_{c}+1)$ if $a\ne 0$ and $\nu_2(2\kappa-{v_{c}})>\nu_2(u_{c}+1)$,
and $\varepsilon_1=0$ otherwise,
\item $\varepsilon_2=1$ if $a=0$, $z\hspace{1mm}\hnd h(z,w)$, and
the lowest degree of $h(0,w)$ in the variable $w$ is even, and $\varepsilon_2=0$ otherwise.
\end{itemize}
\end{theorem}

Recently, Hertling and Mase proved that Orlik's conjecture holds true if a singularity is a Thom-Sebastiani sum of singularities of Brieskorn-Pham, cyclic, or chain type~\cite{HM22}.
See Definition~\ref{definition5-1} for these terminologies.
If $a=0$ in~\eqref{eq1-1}, then, since $(f,O)$ is weighted homogeneous and isolated, we may assume that $h(z,w)=z^pw^q+w^s$, where $p+q\geq 4$, $q\in \{0,1\}$, and $s\geq 4$. Therefore, $(f,O)$ is in the form of a Thom-Sebastiani sum of Brieskorn-Pham, cyclic, or chain type singularities. There are also cases with $a\ne 0$ for which $(f,O)$ has this form.
In such cases, Orlik's argument in~\cite{Orl72} applies, and hence $H_2(K_f)$ can be determined.

\begin{theorem}\label{thm3}
Let $(f,O)$ be a weighted homogeneous, isolated cDV singularity of type $cD_n$ with $n>4$ 
given in the form~\eqref{eq1-1}. 
Assume that $(f,O)$ is a Thom-Sebastiani sum of singularities of Brieskorn-Pham, cyclic, or chain type.
Then, $H_2(K_f)$ is isomorphic to one of the following groups:
the trivial group, $\Z^m$, $\Z_2^{2m}$, and $\Z\oplus \Z_2^{2m}$, where $m\in\N$.
In particular, $K_f$ is diffeomorphic to $S^5$ if  $H_2(K_f)$ is trivial,
and to the connected sum of $m$ copies of $S^2\times S^3$ if $H_2(K_f)$ is $\Z^m$.
\end{theorem}

Consequently, it follows that $H_2(K_f)$ has no torsion except possibly $\Z_2$-factors.
A precise statement of the above theorem will be given in~Theorem~\ref{thm51}.
Note that the second assertion of the above theorem again follows from Smale's result in~\cite{Sma62}.

This paper is organized as follows.
In Section~2, we introduce basic terminologies as Newton non-degeneracy, cDV singularities, and zeta functions of complex hypersurface singularities.
In Section~3, we consider the case of isolated cDV singularities of type $cA_n$.
In Section~4, we focus on cDV singularities of type $cD_n$ with $n>4$.
We first calculate the factors of the zeta functions that correspond to simplices appearing on the Newton boundaries of cDV singularities of type $cD_n$, and then prove Theorem~\ref{thm2} by combining these factors.
In Section~5, we restrict our attention to weighted homogeneous singularities and prove Theorem~\ref{thm3}.

The authors would like to thank Shihoko Ishii, J\'{a}nos Koll\'{a}r,  Osamu Saeki for their valuable suggestions and helpful comments.
The first author is supported by JSPS KAKENHI Grant Numbers JP23K03098 and JP23H00081.
The second author is supported by JSPS KAKENHI Grant Numbers JP25K24558.

\section{Preliminary}

\subsection{Newton non-degeneracy}\label{sec21}

In this subsection, we recall Newton non-degenerate singularities.
Let 
\[
   f(x_1,x_2,\ldots,x_k)=\sum_{(i_1,i_2,\ldots,i_k)\in \Z_{\geq 0}^k}a_{i_1i_2\cdots i_k}x_1^{i_1}x_2^{i_2}\cdots x_k^{i_k}
\]
be a polynomial of $k$ variables such that $f(O)=0$, where $O$ is the origin, $a_{i_1i_2\cdots i_k}\in\C$, and $\Z_{\geq 0}=\{i\in\Z\mid i\geq 0\}$.
Let $(f,O)$ denote the singularity of the polynomial function $f:(\C^k,O)\to(\C,0)$ at $O$.
The {\it Newton polygon} $\Gamma_+(f)$ of $(f,O)$ is defined by the convex hull of 
\[
   \left\{(i_1,i_2,\ldots, i_k)+\R_{\geq 0}^k\mid a_{i_1i_2\cdots i_k}\ne 0\right\}\subset\R^k,
\]
where $\R_{\geq 0}=\{t\in\R\mid t\geq 0\}$.
The {\it Newton boundary} $\Gamma(f)$ of $(f,O)$ is defined by the union of the compact faces of $\Gamma_+(f)$.
A vector $P={}^t(p_1,p_2,\ldots, p_k)$ in $\Z_{\geq 0}^k$ satisfying $\gcd(p_1,p_2,\ldots,p_k)=1$ is called a primitive vector. 
If $p_i>0$ for $i=1,2,\ldots,k$, then it is called a positive primitive vector.
For each primitive vector $P={}^t(p_1,p_2,\ldots, p_k)$,
set 
\[
   d(P;f)=\min\left\{\sum_{i=1}^k p_iX_i\in\R_{\geq 0}\mid (X_1,X_2,\ldots,X_k)\in\Gamma_+(f) \right\}.
\]
Let $\Delta(P;f)$ denote the face of $\Gamma_+(f)$ defined by
\[
   \Delta(P;f)=\left\{(X_1,X_2,\ldots,X_k)\in\Gamma(f) \mid \sum_{i=1}^k p_iX_i=d(P;f)\right\}.
\]
The function
\[
   f_P(x_1,x_2,\ldots,x_k)=\sum_{(i_1,i_2,\ldots,i_k)\in \Delta(P;f)}a_{i_1i_2\cdots i_k}x_1^{i_1}x_2^{i_2}\cdots x_k^{i_k}
\]
is called the {\it face function} of $f$ with respect to $P$.

A singularity $(f,O)$ is said to be {\it Newton non-degenerate} if $f_P:(\C^*)^k\to\C$ has no critical points for any compact face $\Delta(P;f)$ of $\Gamma_+(f)$, where $\C^*=\C\setminus\{0\}$.
Note that if $f$ is a two-variable function, then $(f,O)$ is Newton non-degenerate if and only if, for each $1$-dimensional compact face $\Delta(P;f)$, its face function $f_P$ has no factor with multiplicity greater than $1$.

A singularity $(f,O)$ is said to be {\it weighted homogeneous} if
\[
   f(x_1,x_2,\ldots,x_k)=\sum_{p_1i_1+p_2i_2+\cdots +p_ki_k=d}a_{i_1i_2\cdots i_k}x_1^{i_1}x_2^{i_2}\cdots x_k^{i_k}
\]
for some primitive vector $P={}^t(p_1,p_2,\ldots, p_k)$ and an integer $d$.
In particular, $\Gamma(f)$ is contained in the hyperplane $\{(X_1,X_2,\ldots, X_k)\in\R^k\mid p_1X_1+p_2X_2+\cdots +p_kX_k=d\}$.

\subsection{cDV singularities}

A singularity is called a {\it Du Val} singularity if it is equivalent to one of the following:
\begin{itemize}
\item[$A_n$:\,] $x^2+y^2+z^{n+1}$ with $n\geq 1$\,;
\item[$D_n$:\,] $x^2+y^2z+z^{n-1}$ with $n\geq 4$\,;
\item[$E_6$:\,] $x^2+y^3+z^4$\,;
\item[$E_7$:\,] $x^2+y^3+yz^3$\,;
\item[$E_8$:\,] $x^2+y^3+z^5$.
\end{itemize}
A $3$-dimensional hypersurface singularity $(f,O)$ is called a {\it compound Du Val} singularity
(cDV singularity, for short) if the singularity obtained as the intersection of $(f,O)$ with a generic hyperplane
is a Du Val singularity.
A $3$-dimensional hypersurface singularity over $\C$ is terminal if and only if it is an isolated cDV singularity~\cite{Rei80}.
In the theorem below, $f_d$, $g_d$, and $h_d$ are homogeneous polynomials of degree $d$, and $f_{\geq d}$, $g_{\geq d}$, and $h_{\geq d}$ are polynomials whose lowest degree is at least $d$.

\begin{theorem}[Koll\'{a}r~{\cite[Theorems~2.4, 2.8, 2.9, and 2.10]{Kol98}}]
An isolated cDV singularity is equivalent to one of the following:
\begin{itemize}
\item[$cA_n$:\,] $x^2+y^2+z^2+w^m$ with $m\geq 2$ if $n=1$ and
$x^2+y^2+f_{\geq 3}(z,w)$ if $n>1$.
In the case $n>1$, $n$ is the lowest degree of $ f_{\geq 3}(z,w)$ minus $1$.
\item[$cD_n$:\,] $x^2+f_{\geq 3}(y,z,w)$ with $f_3\ne 0$ if $n=4$
and $x^2+y^2z+ayx^\kappa+h_{\geq d}(z,w)$ if $n>4$,
where $f_3$ is not divisible by the square of a linear form, $a\in \C$, $\kappa\geq 3$, $d\geq 4$,
and $h_d\ne 0$. In the case $n>4$, $n=\min\{2\kappa, d+1\}$ if $a\ne 0$ and $n=d+1$ if $a=0$.
\item[$cE_6$:\,] $x^2+y^3+yg_{\geq 3}(z,w)+h_{\geq 4}(z,w)$, where $h_4\ne 0$.
\item[$cE_7$:\,] $x^2+y^3+yg_{\geq 3}(z,w)+h_{\geq 5}(z,w)$, where $g_3\ne 0$.
\item[$cE_8$:\,] $x^2+y^3+yg_{\geq 4}(z,w)+h_{\geq 5}(z,w)$, where $h_5\ne 0$.
\end{itemize}
\end{theorem}

In this paper, we study cDV singularities of type $cA_n$ with $n\geq 1$ and $cD_n$ with $n>4$.
In the case of type $cD_n$ with $n>4$, we further assume that singularities are Newton non-degenerate. 
This condition is necessary in order to study the topology of the link of the singularity from its Newton boundary.
We exclude the cases $cD_4$ and $cE_n$ ($n=6,7,8$),
since their Newton boundaries may have various shapes, making them more difficult to analyze than the case $cD_n$ with $n>4$.

\subsection{Isolatedness}

We give a necessary and sufficient condition for a Newton non-degenerate cDV singularity of type $cD_n$
to be isolated.
If the term $h(z,w)$ in~\eqref{eq1-1} is identically zero, then $(f,O)$ is non-isolated.
Hereafter, we assume that $h(z,w)$ is not identically zero.

\begin{lemma}\label{lemma21}
Let $(f,O)$ be a cDV singularity given in the form~\eqref{eq1-1}.
Assume that $(f,O)$ is Newton non-degenerate.
\begin{itemize}
\item[(1)] If $a=0$, then $(f,O)$ is isolated if and only if $z\hspace{1mm}\hnd h(z,w)$ and $w^2\hspace{1mm}\hnd h(z,w)$. 
\item[(2)] If $a\ne 0$, then $(f,O)$ is isolated if and only if $w^2\hspace{1mm}\hnd h(z,w)$.
\end{itemize}
\end{lemma}

\begin{proof}
Assume that $(f,O)$ is non-isolated.
Since $\frac{\partial f}{\partial x}=2x$, 
the singularity of $H(y,z,w)=y^2z+ayw^\kappa+h(z,w)$ at the origin is also non-isolated.
The Newton non-degeneracy of $(H,O)$ implies that
the singular locus $\Sing(H)$ of $H$ must lie on one of the coordinate subspaces $y=0$, $z=0$, or $w=0$.

First we prove (1).
A singular point $(y,z,w)$ of $H$ satisfies $2yz=0$, $y^2+\frac{\partial h}{\partial z}(z,w)=0$ and $\frac{\partial h}{\partial w}(z,w)=0$. 
If $\Sing(H)\subset\{y=0\}$, then $(h,O)$ is non-isolated, otherwise $(H,O)$ is isolated.
By the Newton non-degeneracy of $(h,O)$, either $z^2\hd h(z,w)$ or $w^2\hd h(z,w)$ holds.
Suppose that $\Sing(H)\not\subset\{y=0\}$.
Then, $\Sing(H)\subset\{z=0\}$ by $2yz=0$.
If  $z\hspace{1mm}\hnd h(z,w)$ then $h$ is of the form $h(z,w)=zh_1(z,w)+bw^s$, where $b\ne 0$ and $s\geq 4$.
Then, $\frac{\partial h}{\partial w}=z\frac{\partial h_1}{\partial w}+bsw^{s-1}=0$ implies $w=0$ on $\Sing(H)$.
Therefore, $\Sing(H)$ is given by $z=w=0$. 
Since $(h,O)$ is singular, $\frac{\partial h}{\partial z}=\frac{\partial h}{\partial w}=0$ holds on $\Sing(H)$,
which implies that $y^2+\frac{\partial h}{\partial z}(z,w)\ne 0$ for $y\ne 0$.
This is a contradiction.
In summary, if $(f,O)$ is non-isolated, then either $z\hd h(z,w)$ or $w^2\hd h(z,w)$ holds.
On the other hand, it is easy to verify that if either $z\hd h(z,w)$ or  $w^2\hd h(z,w)$ holds, then $(f,O)$ is non-isolated. Thus, the assertion follows in this case.

Next we prove (2).
A singular point $(y,z,w)$ of $H$ satisfies $2yz+aw^\kappa=0$, $y^2+\frac{\partial h}{\partial z}(z,w)=0$ and 
$a\kappa yw^{\kappa-1}+\frac{\partial h}{\partial w}(z,w)=0$. 
If $\Sing(H)\subset\{y=0\}$, then $w=0$ on $\Sing(H)$.
Since $(h,O)$ is Newton non-degenerate and non-isolated, and $\Sing(H)\not\subset\{z=0\}$, it follows that $w^2\hd h(z,w)$.
Suppose that $\Sing(H)\not\subset\{y=0\}$.
If $\Sing(H)\subset\{z=0\}$, it follows from $2yz+aw^\kappa=0$ that $\Sing(H)$ is given by $z=w=0$.
This implies that $y^2+\frac{\partial h}{\partial z}(z,w)\ne 0$ for $y\ne 0$, which is a contradiction.
If $\Sing(H)\not\subset\{z=0\}$, then $\Sing(H)\subset\{w=0\}$.
However, this contradicts $2yz+aw^\kappa=0$.
In summary, if $(f,O)$ is non-isolated, then $w^2\hd h(z,w)$ holds.
On the other hand, if $w^2\hd h(z,w)$ holds, then the set given by $x=y=w=0$ is contained in the singular locus of $f$, and hence $(f,O)$ is non-isolated. This completes the proof.
\end{proof}

\subsection{Zeta function and Newton polygon}

Let $f:\C^k\to\C$ be a polynomial function of $k$ variables such that $f(O)=0$, where $O$ is the origin, and let $(f,O)$ denote the singularity of $f$ at $O$.
The map $\phi:S^{2k-1}_\varepsilon\setminus K_f\to S^1$ defined by $\phi(x)=f(x)/|f(x)|$ is a locally trivial fibration, called the {\it Milnor fibration}~\cite{Mil68}, where $S^{2k-1}_\varepsilon$ is the $(2k-1)$-dimensional sphere centered at $O$ with sufficiently small radius $\varepsilon$. 
To determine the homology group of the link $K_f=S_\varepsilon^{2k-1}\cap f^{-1}(0)$ of $(f,O)$, we use the short exact sequence
\[   
  0\longrightarrow H_{k-1}(K_f)
  \longrightarrow H_{k-1}(F_f)
  \overset{I-T}\longrightarrow H_{k-1}(F_f)
  \longrightarrow H_{k-2}(K_f)
  \longrightarrow 0,
\]
where $F_f$ is the fiber of the Milnor fibration, $T$ is the monodromy matrix of $H_{k-1}(F_f)$, and $I$ is the identity matrix.
The zeta function $\zeta_f(t)$ of $(f,O)$ is defined by
\[
   \zeta_f(t)=\prod_{i=0}^{k-1} \det(I-tT_i)^{(-1)^{i+1}},
\]
where $T_i$ is the monodromy matrix of $H_i(F_f;\Q)$.
In particular, if $(f,O)$ is an isolated singularity, then $H_i(F_f)=0$ for $i\ne 0, k-1$, and therefore the characteristic polynomial of the monodromy matrix of $H_{k-1}(F_f)$ is given by $(1-t)\zeta_f(t)$.

The zeta function $\zeta_f(t)$ can be obtained from a resolution of the singularity $(f,O)$, using A'Campo's formula as follows.
Let $\pi:Y\to \C^k$ be a resolution of $(f,O)$, where $Y$ is a complex manifold and $(\pi\circ f)^{-1}(0)$ is the union of divisors having only normal crossings. Let $\tilde V$ be the strict transform of $V=f^{-1}(0)$, and $E_1,\ldots,E_\ell$ be the set of the exceptional divisors.
For each $i=1,\ldots,\ell$, let $m(E_i)$ be the multiplicity of $\pi\circ f$ along $E_i$, and set $E^{co}_i=\left(E_i\setminus \left(\tilde V\cup \bigcup_{j\ne i}E_j\right)\right)\cap \pi^{-1}(0)$. 

\begin{theorem}[A'Campo~\cite{AC75} (cf.~\cite{Oka97})]\label{thmac}
\[
\zeta_f(t)=\prod_{i=1}^\ell (1-t^{m(E_i)})^{-\chi(E_i^{co})},
\]
where $\chi(E_i^{co})$ is the Euler characteristic of $E_i^{co}$.
\end{theorem}

If $(f,O)$ is Newton non-degenerate, then it admits a resolution with respect to the Newton boundary $\Gamma(f)$, called a {\it toric modification} (cf.~\cite{Oka97}).
Let $(X_1,X_2,\ldots,X_k)$ be the coordinates of the space $\R^k$ on which the Newton polygon $\Gamma_+(f)$ is defined.
Let $I$ be a subset of the index set $\{1,2,\ldots,k\}$, and $X^I$ denote the subspace of $\R^k$ spanned by $(X_i \mid i\in I)$.
For each $I$, let $f^I$ be the restriction of $f$ to the variables $(x_i\mid i\in I)$, and let $S_I$ denote the set of positive primitive vectors $Q$ on $X^I$ such that $\Delta(Q;f^I)$ is a $(|I|-1)$-dimensional compact face of $\Gamma_+(f^I)$, where $|I|$ is the number of elements in $I$.

Let $\mathcal I$ be the set of subsets of $\{1,2,\ldots,k\}$.
The next theorem follows from A'Campo's formula above.
For instance, see~\cite[Theorem (5.3) in Chap. III]{Oka97}.

\begin{theorem}[Varchenko~\cite{Var76}]\label{thmvar}
\[
   \zeta_f(t)=\prod_{I\in \mathcal I} \zeta_I(t), \qquad
   \zeta_I(t)=\prod_{Q\in\mathcal S_I}(1-t^{d(Q;f^I)})^{-\chi(Q)},
\]
where 
\[
\chi(Q)=\frac{(-1)^{|I|-1}|I|!\,\text{\rm Vol}_{|I|}(\text{\rm Cone}(\Delta(Q;f^I),0))}{d(Q;f^I)}.
\]
Here $\text{\rm Vol}_{|I|}(R)$ denotes the Euclidean volume of a closed subset $R\subset X^I$, and $\text{\rm Cone}(\Delta(Q;f^I),0)$ is the cone of $\Delta(Q;f^I)$ with the origin $0\in X^I$.
\end{theorem}

\section{Links of cDV singularities of type $cA_n$}



In this section, we prove Theorem~\ref{thm1}. 
For a matrix $S$, let ${}^tS$ denote the transposition matrix of $S$.
The next lemma is well-known (cf.~\cite{Rol76}). We include a proof for completeness.

\begin{lemma}\label{lemma41}
Let $L$ be a fibered link consisting of $m$ components,
and let $g$ be the genus of the fiber surface $F$ of $L$.
Then, there is a basis $(a_1,a_2,\ldots, a_{2g+m-1})$ of $H_1(F)$ such that the Seifert matrix $S$ satisfies
\[
S-{}^tS=
\begin{pmatrix}
O_{m-1} & O \\ O & R-{}^tR
\end{pmatrix},
\]   
where $O_{m-1}$ is the $(m-1)\times (m-1)$ zero matrix, the other two $O$ are the $(m-1)\times 2g$ and $2g\times (m-1)$ zero matrices, and $R-{}^tR$ is the $2g\times 2g$ block-diagonal matrix with $g$ copies of $\begin{pmatrix} 0 & -1 \\ 1 & 0\end{pmatrix}$ along the diagonal, all other entries being zero.
\end{lemma}

\begin{proof}
Choose a basis $(a_1,a_2,\ldots, a_{m+2g-1})$ of $H_1(F)$ as shown in Figure~\ref{fig1}.
Let $K_1,\ldots, K_m$ be the simple closed curves that constitute the boundary of $F$, oriented as shown in the figure.
Remark that the surface in Figure~\ref{fig1} is non-trivially embedded in $S^3$. In particular, the simple closed curves 
$a_1,\ldots,a_{m+2g-1}, K_1,\ldots,K_{m-1}$ are, in general, mutually linked in $S^3$.

\begin{figure}[htbp]
\includegraphics[scale=1.0, bb=129 613 559 712]{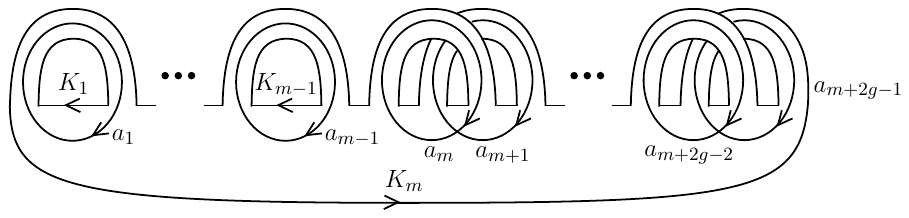}
\caption{A basis on the fiber surface $F$}
\label{fig1}
\end{figure}

Let $lk(c_1,c_2)$ denote the linking number of two disjoint simple closed curves $c_1$ and $c_2$ in $S^3$.
Note that $lk(c_1,c_2)=lk(c_2,c_1)$.
Let $P$ be the linking matrix of $K_1,K_2,\ldots,K_m$, which is a symmetric matrix.
From Figure~\ref{fig1}, we can verify that $S$ is given as
\[
   S=\begin{pmatrix}
   P' & Q \\ {}^tQ & R
   \end{pmatrix},
\]
where $P'$ is the $(m-1)\times (m-1)$ matrix obtained from $P$ by removing the $m$-th low and the $m$-th column, $Q$ is the $(m-1)\times 2g$ matrix whose $(i,j)$-entry is $lk(K_i,a_j)$, and $R$ is the $2g\times 2g$ matrix whose $k$-th block-diagonal is
\[
\begin{pmatrix}
  lk(a_{m+2(k-1)}, a_{m+2(k-1)}) & lk(a_{m+2k-1}, a_{m+2(k-1)})-1 \\
  lk(a_{m+2(k-1)}, a_{m+2k-1}) & lk(a_{m+2k-1}, a_{m+2k-1})
\end{pmatrix},
\]
where $k=1,2,\ldots,g$, and the other entries satisfy $r_{ij}=r_{ji}$, where $r_{ij}$ is the $(i,j)$-entry of $R$.
Since $P'$ is a symmetric matrix, it follows that 
\[
   S-{}^tS
   =\begin{pmatrix}
   P' & Q \\ {}^tQ & R
   \end{pmatrix}
   -\begin{pmatrix}
   P' & Q \\ {}^tQ & R
   \end{pmatrix}
   =\begin{pmatrix}
   O_{m-1} & O \\ O & R-{}^tR
   \end{pmatrix}.
\]
Thus the assertion follows.
\end{proof}

\begin{proof}[Proof of Theorem~\ref{thm1}]
Let $S$ be a Seifert matrix of the fibered link $L=S_\varepsilon^3\cap h^{-1}(0)$ in the $3$-sphere $S_\varepsilon^3$ satisfying the property in Lemma~\ref{lemma41}.
Let $T_h$ be the monodromy matrix of $H_1(F_h)$, where $F_h$ is the fiber surface of $L$, with respect to the basis chosen in Lemma~\ref{lemma41}.
Note that they are related as $T_h=S^{-1}\,{}^tS$.
Then, we have
\[
   I-T_h
   =I-S^{-1}\,{}^tS=S^{-1}(S-{}^tS),
\]
where $I$ is the identity matrix.
Since $\det S=\pm 1$ (cf.~\cite{Mil68, Dur74, Kat74}), we obtain $\text{\rm coker}(I-T_h)\cong \text{\rm coker}(S-{}^tS)$.
Hence, $\text{\rm coker}(I-T_h)\cong\Z^{m-1}$.

Let $(f,O)$ be the singularity in the assertion, $K_f$ be its link, $F_f$ be its Milnor fiber,
$T:H_3(F_f)\to H_3(F_f)$ be the monodromy matrix.
Then, the short exact sequence
\[   
  0\longrightarrow H_3(K_f)
  \longrightarrow H_3(F_f)
  \overset{I-T}\longrightarrow H_3(F_f)
  \longrightarrow H_2(K_f)
  \longrightarrow 0
\]
implies that 
\[
   H_3(K_f)\cong\ker (I-T)\quad\text{and}\quad H_2(K_f)\cong\text{\rm coker}(I-T),
\]
see~\cite{Orl72}. 
By the Join Theorem of Sebastian and Thom~\cite{ST71} (cf.~\cite[Theorem (2.9.1) in Chap.~I]{Oka97}),
$T=(-1)\otimes (-1)\otimes T_h=T_h$ holds for a suitable basis of $H_3(F_f)$.
Thus, we obtain $\text{\rm coker}(I-T)\cong\Z^{m-1}$. 
\end{proof}

\begin{remark}
The proof of Theorem~\ref{thm1} remains valid even if the Milnor fibration of $(h,O)$ is replaced by a fibered link in $S^3$.
For example, there exists a real polynomial map from $\R^4$ to $\R^2$ whose singularity at the origin has a figure-eight knot as its link~\cite{Per81}.
Let us denote this polynomial by $h(z_1,z_2,w_1,w_2)$, where $z_1, z_2, w_1, w_2$ are real variables.
Then the link of the singularity of $f(x,y,z_1, z_2, w_1, w_2)=x^2+y^2+h(z_1,z_2,w_1,w_2)$ at the origin is diffeomorphic to $S^5$, since the figure-eight knot consists of a single simple closed curve.
A topological construction of the join of $x^2+y^2$ and a fibered link in $S^3$ can be found in~\cite{Kau74}.
\end{remark}

\section{Links of cDV singularities of type $cD_n$}

\subsection{Data from the Newton boundary of $h(z,w)$}\label{section4-1}

Let
\[
   P_0=\begin{pmatrix} p_0 \\ q_0 \end{pmatrix}, 
   P_1=\begin{pmatrix} p_1 \\ q_1 \end{pmatrix}, 
\ldots,
   P_N=\begin{pmatrix} p_N \\ q_N \end{pmatrix}
\]
be the positive primitive vectors of the Newton boundary $\Gamma(h)$,
where the indices are given so that $p_{i-1}q_i-p_iq_{i-1}>0$ for $i=1,2,\ldots,N$. See Figure~\ref{fig5}.

\begin{figure}[htbp]
\includegraphics[scale=0.7, bb=155 509 410 710]{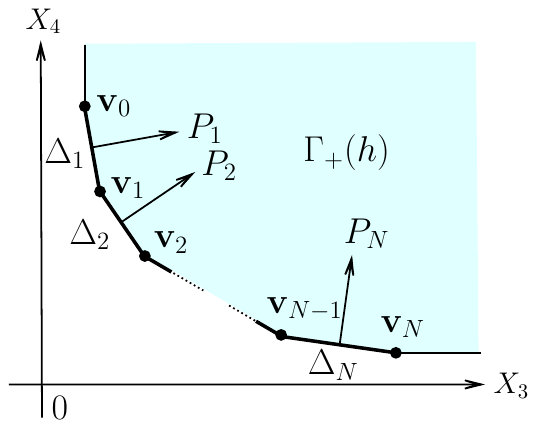}
\caption{The Newton diagram}
\label{fig5}
\end{figure}

Let ${\bf v}_{i-1}=(u_{i-1}, v_{i-1})$ and ${\bf v}_i=(u_i, v_i)$ be the two endpoints of $\Delta(P_i;h)$, where the indices are given so that $u_{i-1}<u_i$ for $i=1,2,\ldots,N$.
In particular, ${\bf v}_i=\Delta(P_i;h)\cap \Delta(P_{i+1};h)$.
Remark that $v_{i-1}>v_i$ holds.

The following data will be used to compute the zeta function of $f$:
\begin{align*}
   \ell_i&=\gcd(u_{i}-u_{i-1},\, v_{i-1}-v_{i}), \\
   m_i&=\begin{vmatrix} u_{i} & u_{i-1} \\ v_{i} & v_{i-1} \end{vmatrix}=u_{i}v_{i-1}-u_{i-1}v_{i}, \\
   m'_i&=\begin{vmatrix} u_{i}-1 & u_{i-1}-1 \\ v_{i} & v_{i-1} \end{vmatrix}=(u_{i}-1)v_{i-1}-(u_{i-1}-1)v_{i}, \\
   M_i&=\gcd(m_i, 2\ell_i), \\
   M_i'&=\gcd(m_i', 2\ell_i), \\
   M_i''&=\gcd(m_i, m_i', 2\ell_i), \\
   n'_i&=\begin{vmatrix} u_{i} & u_{i-1} \\ v_{i}-\kappa & v_{i-1}-\kappa \end{vmatrix}=u_{i}(v_{i-1}-\kappa)-u_{i-1}(v_{i}-\kappa), \\
   N_i'&=\gcd(n_i', \ell_i), \\
   N_i''&=\gcd(m_i, 2n_i', 2\ell_i).
\end{align*}
Note that $\ell_i$ is the lattice length of $\Delta(P_i;h)$.

\begin{lemma}\label{lemma041}
\begin{itemize}
\item[(1)] If $(v_{i-1}-v_i)/\ell_i$ is odd, then either $(M_i, M_i')=(\ell_i, 2\ell_i)$ or $(2\ell_i, \ell_i)$.
\item[(2)] If $(v_{i-1}-v_i)/\ell_i$ is even, then either $(M_i, M_i')=(\ell_i, \ell_i)$ or $(2\ell_i, 2\ell_i)$.
\end{itemize}
\end{lemma}

\begin{proof}
Since $m_i=(u_i-u_{i-1})v_{i-1}-(v_i-v_{i-1})u_{i-1}$ and $\ell_i=\gcd(u_i-u_{i-1}, v_{i-1}-v_i)$, 
it follows that $M_i=\gcd(m_i, 2\ell_i)$ is equal to either $\ell_i$ or $2\ell_i$.
Since $m_i-m_i'=v_{i-1}-v_i$, 
it follows that $M'_i=\gcd(m'_i, 2\ell_i)$ is also equal to either $\ell_i$ or $2\ell_i$.

If $(v_{i-1}-v_i)/\ell_i$ is odd, then $\ell_i\hd (v_{i-1}-v_i)$ and $2\ell_i\hspace{1mm}\hnd (v_{i-1}-v_i)$.
Since $m_i-m_i'=v_{i-1}-v_i$, it follows that $M_i\ne M_i'$. This proves assertion~(1).
If $(v_{i-1}-v_i)/\ell_i$ is even, then $2\ell_i\hd (v_{i-1}-v_i)$.
Since $m_i-m_i'=v_{i-1}-v_i$, it follows that $M_i=\gcd(m_i,2\ell_i)=\gcd(m'_i,2\ell_i)=M'_i$. 
This proves assertion~(2).
\end{proof}

\begin{lemma}\label{lemma041a}
$N_i'=\ell_i$ and $N_i''=M_i$ hold.
\end{lemma}

\begin{proof}
Since $m_i=(u_i-u_{i-1})v_{i-1}-(v_i-v_{i-1})u_{i-1}$, it follows that $\ell_i\hd m_i$.
Hence, from $m_i-n_i'=\kappa(u_i-u_{i-1})$, we obtain $\ell_i\hd n'_i$.
Thus, $N'_i=\gcd(n'_i, \ell_i)=\ell_i$ and $M_i=\gcd(m_i,2\ell_i)=\gcd(m_i,2n'_i,2\ell_i)=N''_i$.
\end{proof}

\begin{remark}\label{rem43}
The inequality $\nu_2(v_{i-1}-v_i)>\nu_2(\ell_i)$ in Theorem~\ref{thm2} holds if and only if $(v_{i-1}-v_i)/\ell_i$ is even.
Note that it is also equivalent to the inequality $\nu_2(v_{i-1}-v_i)>\nu_2(u_i-u_{i-1})$.
\end{remark}

\subsection{Newton boundary of cDV singularities of type $cD_n$}

Let $(f,O)$ be an isolated cDV singularity of type $cD_n$ with $n>4$ given in the form~\eqref{eq1-1}.
In the following, let $(x,y,z,w)$ be the coordinates of $\C^4$. For a function $f(x,y,z,w)$ from $\C^4$ to $\C$, its Newton polygon $\Gamma_+(f)$ is defined as explained in Section~\ref{sec21}. Let $(X_1,X_2,X_3,X_4)$ be the coordinates of the space $\R^4$ on which the Newton polygon $\Gamma_+(f)$ is defined.
For simplicity, write $\Delta_i=\Delta(P_i;h)$.
This is also a face of $\Gamma_+(f)$.
Let ${\bf u}_1=(2,0,0,0)\in \R^4$ be the vertex of $\Gamma_+(f)$ corresponding to the term $x^2$, 
${\bf u}_2=(0,2,1,0)$ the vertex corresponding to $y^2z$, and 
${\bf u}_3=(0,1,0,\kappa)$ the vertex corresponding to $yw^\kappa$.

The following lemma describes the Newton boundary $\Gamma(f)$ of $f(x,y,z,w)$.

\begin{lemma}\label{lemma43z}
There exists an index $c$ in $\{0,\ldots, N-1\}$ such that
the Newton boundary $\Gamma(f)$ is the union of the following $3$-simplices{\rm \hspace{0.2mm}:}
\begin{itemize}
\item[(A)] the convex hull of the edge $\Delta_i$ and the two vertices ${\bf u}_1$ and ${\bf u}_2$ for $i\in \{c+1,\ldots, N\}$\,{\rm ;}
\item[(B)] the convex hull of the four vertices ${\bf v}_{c}$, ${\bf u}_1$, ${\bf u}_2$ and ${\bf u}_3$\,{\rm ;}
\item[(C)] the convex hull of the edge $\Delta_i$ and the two vertices ${\bf u}_1$ and ${\bf u}_3$ 
for $i\in \{1,\ldots, c\}$ in the case $c\geq 1$.
\end{itemize}
\end{lemma}

\begin{proof}
If $a=0$, then $\Gamma(f)$ consists of only faces in (A). Therefore, $c=0$ and the assertion follows.
Suppose that $a\ne 0$.
The primitive vector $Q_i^{(1,2)}$ orthogonal to the convex hull $\Delta_i^{(1,2)}$ of ${\bf u}_1$, ${\bf u}_2$, ${\bf v}_{i-1}$ and ${\bf v}_i$ is
\begin{equation}\label{eqq12}
Q_i^{(1,2)}
=\frac{1}{M_i''}\begin{pmatrix} m_i \\ m_i' \\ 2(v_{i-1}-v_i) \\ 2(u_i-u_{i-1}) \end{pmatrix}.
\end{equation}
Since $d(Q_i^{(1,2)}; f)=2m_i/M_i''$,
it follows that $\Delta_i^{(1,2)}$ is not contained in $\Gamma(f)$ if and only if 
\[
\frac{2m_i}{M_i''}>d(Q_i^{(1,2)}; yw^\kappa)=\frac{m_i'+2\kappa(u_i-u_{i-1})}{M_i''}
=\frac{m_i+(v_i-v_{i-1})+2\kappa(u_i-u_{i-1})}{M_i''}.
\]
This inequality is equivalent to
\[
(u_{i-1}+1)\frac{v_i-v_{i-1}}{u_i-u_{i-1}}-v_{i-1}< -2\kappa.
\]
The first term of the left-hand side is a positive multiple of the slope of $\Delta_i$, which is monotonically increasing in $i$.
The second term, $-v_{i-1}$, also increases in $i$.
Then, the largest index satisfying this inequality is the index $c$ in the assertion.
\end{proof}

The faces of $\Gamma(f)$ lying in coordinate subspaces of $\R^4$ are decomposed into the following six cases:
\begin{itemize}
\item[(a)] faces containing $\Delta_i$ in Case (A);
\item[(b)] faces containing the two vertices ${\bf u}_2$ and ${\bf u}_3$;
\item[(c)] faces containing $\Delta_i$ in Case (C);
\item[(d)] faces containing ${\bf v}_N\cap X_3$;
\item[(e)] faces containing ${\bf v}_0\cap X_4$;
\item[(f)] the vertex ${\bf u}_1$.
\end{itemize}
Note that no faces exist in Cases~(b) and~(c) if $a=0$, and no faces exist in Case~(d) if ${\bf v}_N\cap X_3$ is empty, nor in Case~(e) if  ${\bf v}_0\cap X_4$ is empty.
In the following subsections, we determine the factors of $\zeta_f(t)$ associated with each case.

\subsection{Factors from the faces in Case~(a)}

Let $\Delta_i$ be a face of $\Gamma(h)$ in Case~(a).
The faces of $\Gamma(f)$ containing $\Delta_i$ and lying in a coordinate subspace of $\R^4$ are
the following (see Figure~\ref{fig2}):
\begin{itemize}
\item the convex hull of ${\bf u}_1$, ${\bf u}_2$, and $\Delta_i$, which is a $3$-simplex, denoted by $\Delta_i^{(1,2)}$;
\item the convex hull of ${\bf u}_1$ and $\Delta_i$, which is a $2$-simplex lying in the subspace spanned by $(X_1,X_3,X_4)$, denoted by $\Delta_i^{(1)}$; 
\item the convex hull of ${\bf u}_2$ and $\Delta_i$, which is a $2$-simplex lying in the subspace spanned by $(X_2,X_3,X_4)$, denoted by $\Delta_i^{(2)}$; 
\item the $1$-simplex $\Delta_i$, which lies in the subspace spanned by $(X_3,X_4)$.
\end{itemize}

\begin{figure}[htbp]
\includegraphics[scale=0.9, bb=129 571 507 710]{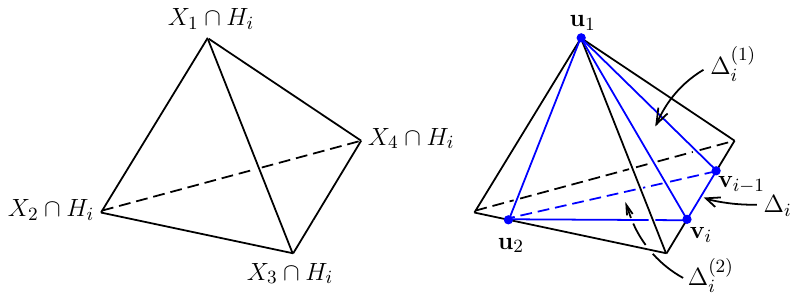}
\caption{The simplex in the left figure lies in the hyperplane $H_i\subset\R_{\geq 0}^4$ containing the face $\Delta_i^{(1,2)}$.
Its vertices are the intersections of $H_i$ with the coordinate subspaces $X_1$, $X_2$, $X_3$, and $X_4$.
The simplex on the right, given as the convex hull of the vertices ${\bf u}_1$, ${\bf u}_2$, ${\bf v}_i$, and ${\bf v}_{i-1}$, is the face $\Delta_i^{(1,2)}$, which lies in $H_i$.}
\label{fig2}
\end{figure}

\begin{lemma}\label{lemma43a}
The product $\zeta_{f,\Delta_i}(t)$ of the factors of $\zeta_f(t)$ corresponding to the faces in Case~(a)
is given as follows:
\[
   \zeta_{f,\Delta_i}(t)
=
\begin{cases}
\displaystyle 
\frac{(1-t^{2m_i/\ell_i})^{\ell_i}}{(1-t^{m_i/\ell_i})^{\ell_i}}
&
\text{if $(v_{i-1}-v_i)/\ell_i$ is odd,} 
\vspace{3mm}
\\
   (1-t^{m_i/\ell_i})^{\ell_i}
& \text{if $(v_{i-1}-v_i)/\ell_i$ is even.}
\end{cases}
\]
\end{lemma}

\begin{proof}
We calculate the factors of $\zeta_f(t)$ corresponding to the faces containing $\Delta_i$ one by one.
First, calculate the factor for $\Delta_i^{(1,2)}$.
The volume $4!\,\text{\rm Vol}_4(\text{\rm Cone}(\Delta_i^{(1,2)},0))$ in the formula in Theorem~\ref{thmvar}
is 
\[
\displaystyle 
\begin{vmatrix} 
\;2 & 0 & 0 & 0\; \\
\;0 & 2 & 0 & 0\; \\
\;0 & 1 & u_i & u_{i-1}\; \\
\;0 & 0 & v_i & v_{i-1}\;
\end{vmatrix}=4m_i.
\]
The primitive vector $Q_i^{(1,2)}$ orthogonal to $\Delta_i^{(1,2)}$ is given in~\eqref{eqq12}.
Note that $\Delta(Q_i^{(1,2)};f)=\Delta_i^{(1,2)}$. Since
\[
d(Q_i^{(1,2)}; f)=\frac{2m_i}{M_i''}
\quad \text{and} \quad
\chi(Q_i^{(1,2)})=\frac{(-1)^{4-1} 4m_i}{d(Q_i^{(1,2)};f)}=-2M_i'',
\]
the factor is $\displaystyle (1-t^{2m_i/M_i''})^{2M_i''}$ by Theorem~\ref{thmvar}.

The factors for the remaining faces can be obtained similarly.
For $\Delta_i^{(1)}$,
since the volume is $3!\,\text{\rm Vol}_3(\text{\rm Cone}(\Delta_i^{(1)},0))=2m_i$ and the primitive vector $Q_i^{(1)}$ orthogonal to $\Delta_i^{(1)}$ in the subspace spanned by $(X_1,X_3,X_4)$ is 
\[
Q_i^{(1)}
=\frac{1}{M_i}\begin{pmatrix} m_i \\ 2(v_{i-1}-v_i) \\ 2(u_i-u_{i-1}) \end{pmatrix},
\]
it follows that 
\[
d(Q_i^{(1)};f(x,0,z,w))=\frac{2m_i}{M_i}
\quad \text{and} \quad
\chi(Q_i^{(1)})=\frac{(-1)^{3-1} 2m_i}{d(Q_i^{(1)};f(x,0,z,w))}=M_i.
\]
Hence the factor is $(1-t^{2m_i/M_i})^{-M_i}$.
For $\Delta_i^{(2)}$, since $3!\,\text{\rm Vol}_3(\text{\rm Cone}(\Delta_i^{(2)},0))=2m_i$ and the primitive vector $Q_i^{(2)}$ orthogonal to $\Delta_i^{(2)}$ is 
\[
Q_i^{(2)}
=\frac{1}{M_i'}\begin{pmatrix} m_i' \\ 2(v_{i-1}-v_i) \\ 2(u_i-u_{i-1}) \end{pmatrix},
\]
it follows that 
\[
d(Q_i^{(2)};f(0,y,z,w))=\frac{2m_i}{M_i'}
\quad \text{and} \quad
\chi(Q_i^{(2)})=\frac{(-1)^{3-1} 2m_i}{d(Q_i^{(2)};f(0,y,z,w))}=M_i'.
\]
Hence the factor is $(1-t^{2m_i/M_i'})^{-M_i'}$.
For $\Delta_i$, since $2!\,\text{\rm Vol}_2(\text{\rm Cone}(\Delta_i,0))=m_i$ and the primitive vector $P_i$ orthogonal to $\Delta_i$ is 
\[
P_i=\frac{1}{\ell_i}\begin{pmatrix} v_{i-1}-v_i \\ u_i-u_{i-1} \end{pmatrix},
\]
it follows that
\[
d(P_i;h)=\frac{m_i}{\ell_i}
\quad \text{and} \quad
\chi(P_i)=\frac{(-1)^{2-1} m_i}{d(P_i;h)}=-\ell_i.
\]
Hence the factor is $(1-t^{m_i/\ell_i})^{\ell_i}$.

The product $\zeta_{f,\Delta_i}(t)$ of the four factors obtained above is
\[
 \zeta_{f,\Delta_i}(t)
=\frac{(1-t^{2m_i/M_i''})^{2M''_i}(1-t^{m_i/\ell_i})^{\ell_i}}{(1-t^{2m_i/M_i})^{M_i}(1-t^{2m_i/M_i'})^{M_i'}}.
\]
If $(v_{i-1}-v_i)/\ell_i$ is odd, then by Lemma~\ref{lemma041}~(1) and the fact that $M_i''=\gcd(M_i, M_i')$, we have $(M_i, M_i', M_i'')=(\ell_i, 2\ell_i, \ell_i)$ or $(2\ell_i, \ell_i, \ell_i)$.
Hence the formula for $\zeta_{f,\Delta_i}(t)$ in the assertion follows.
If $(v_{i-1}-v_i)/\ell_i$ is even, then we have $M_i=M_i'=M_i''$ by Lemma~\ref{lemma041}~(2),
and hence the formula for $\zeta_{f,\Delta_i}(t)$ in the assertion follows.
\end{proof}

\subsection{Factors from the faces in Case (b)}

The faces of $\Gamma(f)$ belonging to Case~(b) and lying in a coordinate subspace of $\R^4$ are
the following (see Figure~\ref{fig7}):
\begin{itemize}
\item the convex hull of ${\bf u}_1$, ${\bf u}_2$, ${\bf u}_3$ and ${\bf v}_{c}$, which is a $3$-simplex, denoted by ${\bf v}_{c}^{(1,2,3)}$;
\item the convex hull of ${\bf u}_2$, ${\bf u}_3$ and ${\bf v}_{c}$, which is a $2$-simplex lying in the subspace spanned by $(X_2,X_3,X_4)$, denoted by ${\bf v}_{c}^{(2,3)}$.
\end{itemize}

\begin{figure}[htbp]
\includegraphics[scale=0.70, bb=129 521 306 707]{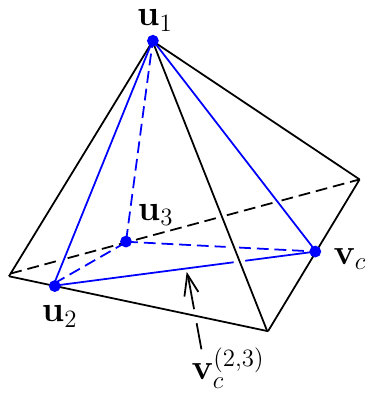}
\caption{Faces in Case~(b)}
\label{fig7}
\end{figure}

Put $M=\gcd(2\kappa-v_{c}, u_{c}+1)$ and $M'=\gcd(2\kappa-v_{c}, 2(u_{c}+1))$.
Then, either $M'=M$ or $M'=2M$ holds, since they satisfy $M\hd M'$ and $M'\hd 2M$.

\begin{lemma}\label{lemma43b}
The product $\zeta_{f,{\bf v}_{c}}(t)$ of the factors of $\zeta_f(t)$ corresponding to the faces in Case~(b)
is given as follows:
\[
   \zeta_{f,{\bf v}_{c}}(t)
=
\begin{cases}
\displaystyle 
\frac{(1-t^{2(2\kappa u_{c}+v_{c})/M})^{M}}{(1-t^{(2\kappa u_{c}+v_{c})/M})^{M}}
&
\text{if $M'=M$,} 
\vspace{3mm}
\\
  (1-t^{(2\kappa u_{c}+v_{c})/M})^{M}
   & \text{if $M'=2M$.}
\end{cases}
\]
\end{lemma}

\begin{proof}
For ${\bf v}_{c}^{(1,2,3)}$, since $4!\,\text{\rm Vol}_4(\text{\rm Cone}({\bf v}_{c}^{(1,2,3)},0))=2(2\kappa u_{c}+v_{c})$
and the primitive vector $Q_{{\bf v}_{c}}^{(1,2,3)}$ orthogonal to ${\bf v}_{c}^{(1,2,3)}$ is 
\[
Q_{{\bf v}_{c}}^{(1,2,3)}
=\frac{1}{M'}\begin{pmatrix} 2\kappa u_{c}+v_{c} \\ (2\kappa u_{c}+v_{c})-(2\kappa-v_{c}) \\ 2(2\kappa-v_{c}) \\ 2(u_{c}+1) \end{pmatrix},
\]
it follows that
\[
d(Q_{{\bf v}_{c}}^{(1,2,3)}; f)=\frac{2(2\kappa u_{c}+v_{c})}{M'}
\quad \text{and} \quad
\chi(Q_{{\bf v}_{c}}^{(1,2,3)})=\frac{(-1)^{4-1} 2(2\kappa u_{c}+v_{c})}{d(Q_{{\bf v}_{c}}^{(1,2,3)};f)}=-M'.
\]
Here we used
\[
\begin{split}
   M'&=\gcd(2\kappa-v_{c}, 2(u_{c}+1))
     =\gcd(2\kappa(u_{c}+1), 2\kappa-v_{c}, 2(u_{c}+1)) \\
   &=\gcd(2\kappa u_{c}+v_{c}, 2\kappa-v_{c}, 2(u_{c}+1)).
\end{split}
\]
Hence the factor is $\displaystyle (1-t^{2(2\kappa u_{c}+v_{c})/M'})^{M'}$ by Theorem~\ref{thmvar}.
For ${\bf v}_{c}^{(2,3)}$, since the volume is $3!\,\text{\rm Vol}_3(\text{\rm Cone}({\bf v}_{c}^{(2,3)},0))=2\kappa u_{c}+v_{c}$ and the primitive vector $Q_{{\bf v}_{c}}^{(2,3)}$ orthogonal to ${\bf v}_{c}^{(2,3)}$ is 
\[
Q_{{\bf v}_{c}}^{(2,3)}
=\frac{1}{M}\begin{pmatrix} \kappa(u_{c}+1)-(2\kappa-v_{c}) \\ 2\kappa-v_{c} \\ u_{c}+1 \end{pmatrix},
\]
it follows that 
\[
d(Q_{{\bf v}_{c}}^{(2,3)};f(0,y,z,w))=\frac{2\kappa u_{c}+v_{c}}{M}
\quad \text{and} \quad
\chi(Q_{{\bf v}_{c}}^{(2,3)})=\frac{(-1)^{3-1} (2\kappa u_{c}+v_{c})}{d(Q_{{\bf v}_{c}}^{(2,3)};f(0,y,z,w))}=M.
\]
Hence the factor is $(1-t^{(2\kappa u_{c}+v_{c})/M})^{-M}$.

The product $\zeta_{f,{\bf v}_{c}}(t)$ of the factors obtained above is
\[
 \zeta_{f,{\bf v}_{c}}(t)
  =\frac{(1-t^{2(2\kappa u_{c}+v_{c})/M'})^{M'}}{(1-t^{(2\kappa u_{c}+v_{c})/M})^{M}}.
\]
Thus, the assertion follows.
\end{proof}

\begin{remark}\label{rem4-7}
The inequality $\nu_2(2\kappa-v_c)>\nu_2(u_c+1)$ in Theorem~\ref{thm2} holds if and only if $M'=2M$,
since $M=\gcd(2\kappa-v_{c}, u_{c}+1)$ and $M'=\gcd(2\kappa-v_{c}, 2(u_{c}+1))$.
\end{remark}

\subsection{Factors from the faces in Case (c)}

Let $\Delta_i$ be a face of $\Gamma(h)$ in Case~(c).
The faces of $\Gamma(f)$ containing $\Delta_i$ and lying in a coordinate subspace of $\R^4$ are the following (see Figure~\ref{fig8}):
\begin{itemize}
\item the convex hull of ${\bf u}_1$, ${\bf u}_3$, and $\Delta_i$, which is a $3$-simplex, denoted by $\Delta_i^{(1,3)}$;
\item the convex hull of ${\bf u}_1$ and $\Delta_i$, which is a $2$-simplex lying in the subspace spanned by $(X_1,X_3,X_4)$, denoted by $\Delta_i^{(1)}$; 
\item the convex hull of ${\bf u}_3$ and $\Delta_i$, which is a $2$-simplex lying in the subspace spanned by $(X_2,X_3,X_4)$, denoted by $\Delta_i^{(3)}$; 
\item the $1$-simplex $\Delta_i$, which lies in the subspace spanned by $(X_3,X_4)$.
\end{itemize}

\begin{figure}[htbp]
\includegraphics[scale=0.95, bb=129 571 283 707]{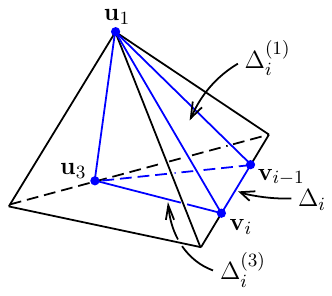}
\caption{Faces in Case~(c)}
\label{fig8}
\end{figure}

\begin{lemma}\label{lemma43c}
No factors corresponding to the faces in Case~(c) remain in $\zeta_f(t)$.
\end{lemma}

\begin{proof}
For $\Delta_i^{(1,3)}$,
since $4!\,\text{\rm Vol}_4(\text{\rm Cone}(\Delta_i^{(1,3)},0))=2m_i$
and the primitive vector $Q_i^{(1,3)}$ orthogonal to $\Delta_i^{(1,3)}$ is 
\[
Q_i^{(1,3)}
=\frac{1}{N_i''}\begin{pmatrix} m_i \\ 2n_i' \\ 2(v_{i-1}-v_i) \\ 2(u_i-u_{i-1}) \end{pmatrix},
\]
it follows that
\[
d(Q_i^{(1,3)}; f)=\frac{2m_i}{N_i''}
\quad \text{and} \quad
\chi(Q_i^{(1,3)})=\frac{(-1)^{4-1} 2m_i}{d(Q_i^{(1,3)};f)}=-N_i''.
\]
Hence the factor is $\displaystyle (1-t^{2m_i/N_i''})^{N_i''}$ by Theorem~\ref{thmvar}.
The calculation of the factor for $\Delta_i^{(1)}$ is exactly same as the one in the proof of Lemma~\ref{lemma43a}. Therefore, the corresponding factor is $(1-t^{2m_i/M_i})^{-M_i}$.
For $\Delta_i^{(3)}$, since $3!\,\text{\rm Vol}_3(\text{\rm Cone}(\Delta_i^{(3)},0))=m_i$ and the primitive vector $Q_i^{(3)}$ orthogonal to $\Delta_i^{(3)}$ is 
\[
Q_i^{(3)}
=\frac{1}{N_i'}\begin{pmatrix} n_i' \\ v_{i-1}-v_i \\ u_i-u_{i-1} \end{pmatrix},
\]
it follows that 
\[
d(Q_i^{(3)};f(0,y,z,w))=\frac{m_i}{N_i'}
\quad \text{and} \quad
\chi(Q_i^{(3)})=\frac{(-1)^{3-1} m_i}{d(Q_i^{(3)};f(0,y,z,w))}=N_i'.
\]
Hence the factor is $(1-t^{m_i/N_i'})^{-N_i'}$.
The calculation of the factor for $\Delta_i$ is again exactly same as the one in the proof of Lemma~\ref{lemma43a}. Therefore, the corresponding factor is $(1-t^{m_i/\ell_i})^{\ell_i}$.

The product $\zeta_{f,\Delta_i}(t)$ of the four factors obtained above is
\[
 \zeta_{f,\Delta_i}(t)
=\frac{(1-t^{2m_i/N_i''})^{N''_i}(1-t^{m_i/\ell_i})^{\ell_i}}{(1-t^{2m_i/M_i})^{M_i}(1-t^{m_i/N_i'})^{N_i'}}.
\]
By Lemma~\ref{lemma041a}, $N_i'=\ell_i$ and $N_i''=M_i$ hold.
Therefore, we obtain $\zeta_{f,\Delta_i}(t)=1$.
\end{proof}

\subsection{Factors from the faces in Case~(d)}

Assume that ${\bf v}_N\cap X_3$ is non-empty. In this case, the following faces of $\Gamma(f)$ 
contain ${\bf v}_N\cap X_3$ and lie on the subspace spanned by $(X_1,X_2,X_3)$ (see~Figure~\ref{fig3}):
\begin{itemize}
\item the convex hull of ${\bf u}_1$, ${\bf u}_2$, and ${\bf v}_N$,  which is a $2$-simplex, denoted by ${\bf v}_N^{(1,2)}$;
\item the convex hull of ${\bf u}_1$ and ${\bf v}_N$,  which is a $1$-simplex, denoted by ${\bf v}_N^{(1)}$;
\item the convex hull of ${\bf u}_2$ and ${\bf v}_N$,  which is a $1$-simplex, denoted by ${\bf v}_N^{(2)}$;
\item the vertex ${\bf v}_N$.
\end{itemize}

\begin{figure}[htbp]
\includegraphics[scale=0.65, bb=129 516 362 707]{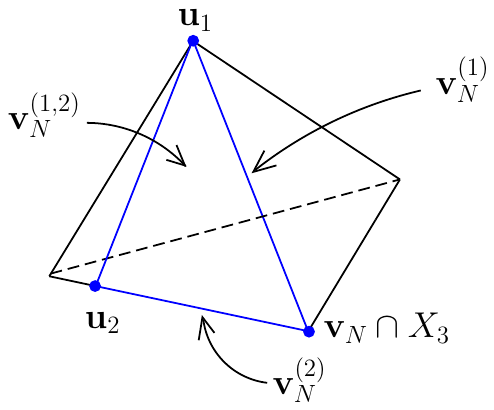}
\caption{Faces in Case~(d)}
\label{fig3}
\end{figure}

\begin{lemma}\label{lemma44}
Assume that ${\bf v}_N\cap X_3$ is non-empty.
The product $\zeta_{f,{\bf v}_N\cap X_3}(t)$ of the factors of $\zeta_f(t)$ corresponding to the faces in Case~(d) is 
\[
   \zeta_{f,{\bf v}_N\cap X_3}(t)
=\frac{1-t^{u_N}}{1-t^{2u_N}}.
\]
\end{lemma}

\begin{proof}
For ${\bf v}_N^{(1,2)}$, since $3!\,\text{\rm Vol}_3(\text{\rm Cone}({\bf v}_N^{(1,2)},0))=4u_N$ and the primitive vector $Q_{{\bf v}_N}^{(1,2)}$ orthogonal to ${\bf v}_N^{(1,2)}$ is 
\[
\displaystyle 
Q_{{\bf v}_N}^{(1,2)}
=\begin{pmatrix} u_N \\ u_N-1 \\ 2 \end{pmatrix},
\]
it follows that 
\[
d(Q_{{\bf v}_N}^{(1,2)};f(x,y,z,0))=2u_N
\quad\text{and}\quad
\chi(Q_{{\bf v}_N}^{(1,2)})=\frac{(-1)^{3-1} 4u_N}{d(Q_{{\bf v}_N}^{(1,2)};f(x,y,z,0))}=2.
\]
Hence the factor is $(1-t^{2u_N})^{-2}$.
For ${\bf v}_N^{(1)}$, since $2!\,\text{\rm Vol}_2(\text{\rm Cone}({\bf v}_N^{(1)},0))=2u_N$ and the primitive vector $Q_{{\bf v}_N}^{(1)}$ orthogonal to ${\bf v}_N^{(1)}$ is 
\[
\displaystyle
Q_{{\bf v}_N}^{(1)}
=\frac{1}{\gcd(u_N,2)}\begin{pmatrix} u_N \\ 2 \end{pmatrix},
\]
it follows that
\[
d(Q_{{\bf v}_N}^{(1)};f(x,0,z,0))=\frac{2u_N}{\gcd(u_N,2)}
\quad \text{and} \quad
\chi(Q_{{\bf v}_N}^{(1)})=\frac{(-1)^{2-1} 2u_N}{d(Q_{{\bf v}_N}^{(1)};f(x,0,z,0))}=-\gcd(u_N,2).
\]
Hence the factor is $(1-t^{2u_N/\gcd(u_N,2)})^{\gcd(u_N,2)}$.
For ${\bf v}_N^{(2)}$, since $2!\,\text{\rm Vol}_2(\text{\rm Cone}({\bf v}_N^{(2)},0))=2u_N$ and the primitive vector $Q_{{\bf v}_N}^{(2)}$ orthogonal to ${\bf v}_N^{(2)}$ is 
\[
\displaystyle 
Q_{{\bf v}_N}^{(2)}
=\frac{1}{\gcd(u_N-1,2)}\begin{pmatrix} u_N-1 \\ 2 \end{pmatrix},
\]
it follows that
\[
d(Q_{{\bf v}_N}^{(2)}
;f(0,y,z,0))=\frac{2u_N}{\gcd(u_N-1,2)}
\quad \text{and}\quad
\chi(Q_{{\bf v}_N}^{(2)})=\frac{(-1)^{2-1} 2u_N}{d(Q_{{\bf v}_N}^{(2)};f(0,y,z,0))}=-\gcd(u_N-1,2).
\]
Hence the factor is $(1-t^{2u_N/\gcd(u_N-1,2)})^{\gcd(u_N-1,2)}$.
The factor corresponding to the vertex ${\bf v}_N\cap X_3$ is $(1-t^{u_N})^{-1}$.

The product $\zeta_{f,{\bf v}_N\cap X_3}(t)$ of the four factors obtained above is
\[
   \zeta_{f,{\bf v}_N\cap X_3}(t)
=\frac{(1-t^{2u_N/\gcd(u_N,2)})^{\gcd(u_N,2)}(1-t^{2u_N/\gcd(u_N-1,2)})^{\gcd(u_N-1,2)}}{(1-t^{2u_N})^2(1-t^{u_N})}
=\frac{1-t^{u_N}}{1-t^{2u_N}},
\]
where we used the fact that $(\gcd(u_N,2), \gcd(u_N-1,2))$ is either $(1,2)$ or $(2,1)$.
Thus, the assertion follows.
\end{proof}

\subsection{Factors from the faces in Case (e)}

Assume that ${\bf v}_0\cap X_4$ is non-empty. 
In this case, the following faces of $\Gamma(f)$ contain ${\bf v}_0\cap X_4$ and lie on the subspace spanned by $(X_1,X_2,X_4)$ (see~Figure~\ref{fig4}):
\begin{itemize}
\item the convex hull of ${\bf u}_1$, ${\bf u}_3$ and ${\bf v}_0$ in the case $a\ne 0$, which is a $2$-simplex, denoted by ${\bf v}_0^{(1,3)}$;
\item the convex hull of ${\bf u}_1$ and ${\bf v}_0$, which is a $1$-simplex, denoted by ${\bf v}_0^{(1)}$;
\item the convex hull of ${\bf u}_3$ and ${\bf v}_0$ in the case $a\ne 0$, which is a $1$-simplex, denoted by ${\bf v}_0^{(3)}$;
\item the vertex ${\bf v}_0$.
\end{itemize}

\begin{figure}[htbp]
\includegraphics[scale=0.65, bb=129 549 363 707]{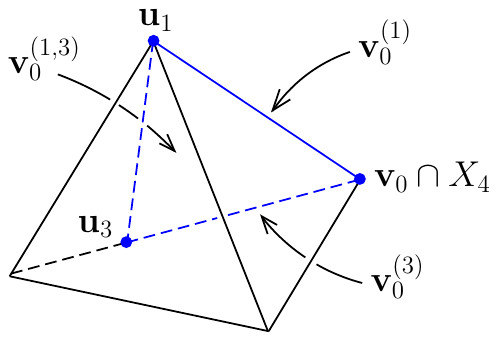}
\caption{Faces in Case~(e)}
\label{fig4}
\end{figure}

\begin{lemma}\label{lemma43}
Assume that ${\bf v}_0\cap X_4$ is non-empty.
The product $\zeta_{f,{\bf v}_0\cap X_4}(t)$ of the factors of $\zeta_f(t)$ corresponding to the faces in Case~(e) is given as follows:
\[
   \zeta_{f,{\bf v}_0\cap X_4}(t)
=
\begin{cases}
\displaystyle 
\frac{1-t^{2v_0}}{1-t^{v_0}} & \text{if $a=0$ and $v_0$ is odd,} 
\vspace{3mm}
\\
\vspace{3mm}
\displaystyle 
1-t^{v_0} & \text{if $a=0$ and $v_0$ is even,} 
\\
1 & \text{if $a\ne 0$.}
\end{cases}
\]
\end{lemma}

\begin{proof}
For ${\bf v}_0^{(1,3)}$, since $3!\,\text{\rm Vol}_3(\text{\rm Cone}({\bf v}_0^{(1,3)},0))=2v_0$ and the primitive vector $Q_{{\bf v}_0}^{(1,3)}$ orthogonal to ${\bf v}_0^{(1,3)}$ is 
\[
Q_{{\bf v}_0}^{(1,3)}
=\frac{1}{\gcd(v_0,2)}\begin{pmatrix} v_0 \\ 2(v_0-\kappa) \\ 2 \end{pmatrix},
\]
it follows that 
\[
d(Q_{{\bf v}_0}^{(1,3)};f(x,y,0,w))=\frac{2v_0}{\gcd(v_0,2)}
\quad\text{and}\quad
\chi(Q_{{\bf v}_0}^{(1,3)})=\frac{(-1)^{3-1} 2v_0}{d(Q_{{\bf v}_0}^{(1,3)};f(x,y,0,w))}=\gcd(v_0,2).
\]
Hence the factor is $(1-t^{2v_0/\gcd(v_0,2)})^{-\gcd(v_0,2)}$.
For ${\bf v}_0^{(1)}$, since $2!\,\text{\rm Vol}_2(\text{\rm Cone}({\bf v}_0^{(1)},0))=2v_0$ and the primitive vector $Q_{{\bf v}_0}^{(1)}$ orthogonal to ${\bf v}_0^{(1)}$ is 
\[
Q_{{\bf v}_0}^{(1)}
=\frac{1}{\gcd(v_0,2)}\begin{pmatrix} v_0 \\ 2 \end{pmatrix},
\]
it follows that 
\[
d(Q_{{\bf v}_0}^{(1)};f(x,0,0,w))=\frac{2v_0}{\gcd(v_0,2)}
\quad\text{and}\quad
\chi(Q_{{\bf v}_0}^{(1)})=\frac{(-1)^{2-1} 2v_0}{d(Q_{{\bf v}_0}^{(1)};f(x,0,0,w))}=-\gcd(v_0,2).
\]
Hence the factor is $(1-t^{2v_0/\gcd(v_0,2)})^{\gcd(v_0,2)}$.
For ${\bf v}_0^{(3)}$, since $2!\,\text{\rm Vol}_2(\text{\rm Cone}({\bf v}_0^{(3)},0))=v_0$ and the primitive vector $Q_{{\bf v}_0}^{(3)}$ orthogonal to ${\bf v}_0^{(3)}$ is 
\[
Q_{{\bf v}_0}^{(3)}
=\begin{pmatrix} v_0-\kappa \\ 1 \end{pmatrix},
\]
it follows that
\[
d(Q_{{\bf v}_0}^{(3)};f(0,y,0,w))=v_0
\quad\text{and}\quad
\chi(Q_{{\bf v}_0}^{(3)})=\frac{(-1)^{2-1} v_0}{d(Q_{{\bf v}_0}^{(3)};f(0,y,0,w))}=-1.
\]
Hence the factor is $1-t^{v_0}$.
The factor corresponding to the vertex ${\bf v}_0\cap X_4$ is $(1-t^{v_0})^{-1}$.
Thus, the assertion follows.
\end{proof}

\subsection{Factors from the face in Case (f)}\label{sec4-8}

The factor of $\zeta_f(T)$ corresponding to the vertex ${\bf u}_1$ is $(1-t^2)^{-1}$.

\subsection{Proof of Theorem~\ref{thm2}}

\begin{proof}[Proof of Theorem~\ref{thm2}]
Let $F_f$ be the Milnor fiber of the singularity $(f,O)$,
$T$ be the monodromy matrix of $H_3(F_f)$,
and $I$ be the identity matrix.
Then, the short exact sequence
\[   
  0\longrightarrow H_3(K_f)
  \longrightarrow H_3(F_f)
  \overset{I-T}\longrightarrow H_3(F_f)
  \longrightarrow H_2(K_f)
  \longrightarrow 0
\]
implies that 
\[
   H_3(K_f)\cong\ker (I-T)\quad\text{and}\quad H_2(K_f)\cong\text{\rm coker}(I-T),
\]
as mentioned in the proof of Theorem~\ref{thm1}.
In particular, the rank of $H_2(K_f)$ equals to the multiplicity of the factor $(1-t)$ in the characteristic polynomial $\det(I-tT)$. Note that $\det(I-tT)=(1-t)\zeta_f(t)$.

The contributions of the factors of $\zeta_f(t)$ to the multiplicity of $(1-t)$ are as follows:
\begin{itemize}
\item In Case (a), the contribution is $\sum_{i}\ell_i$, 
where the summation runs over all indices $i$ from $c+1$ to $N$ such that $(v_{i-1}-v_i)/\ell_i$ is even by Lemma~\ref{lemma43a}.
\item In Case (b), the contribution is $M=\gcd(2\kappa-v_{c}, u_{c}+1)$ if $a\ne 0$ and $M'=2M$ by Lemma~\ref{lemma43b}.
\item In Case (e), the contribution is $1$ if $a=0$, ${\bf v}_0\cap X_4$ is non-empty, and $v_0$ is even by Lemma~\ref{lemma43}.
\item In Case (f), the contribution is $-1$ as mentioned in Section~\ref{sec4-8}.
\item There are no other contributions by these lemmas and Lemmas~\ref{lemma43c} and~\ref{lemma44}.
\end{itemize}
The condition in Case~(a) is
equivalent to the condition on the indices in the summation in Theorem~\ref{thm2}, see Remark~\ref{rem43}.
The condition on $\varepsilon_1$ in Theorem~\ref{thm2} arises from Case~(b).
As mentioned in Remark~\ref{rem4-7},
$M'=2M$ if and only if $\nu_2(2\kappa-{v_{c}})>\nu_2(u_c+1)$.
Hence, we obtain the condition on $\varepsilon_1$ stated in Theorem~\ref{thm2}.
The condition on $\varepsilon_2$ in Theorem~\ref{thm2} arises from Case~(e).

Adding the contribution $1$ of the factor $(1-t)$ in $\det(I-tT)=(1-t)\zeta_f(t)$, we obtain the formula in the assertion.
\end{proof}

\begin{remark}\label{rem410}
Since $M=\gcd(2\kappa-v_{c}, u_{c}+1)$ and
\[
\begin{split}
   M'&=\gcd(2\kappa-v_{c}, 2(u_{c}+1))
   =\gcd(2\kappa u_c+v_{c}, 2\kappa-v_c, 2(u_{c}+1)) \\
   &=\gcd((2\kappa-v_{c})u_c+(u_c+1)v_c, 2\kappa-v_c, 2(u_{c}+1))
   =\gcd((u_c+1)v_c, 2\kappa-v_c, 2(u_{c}+1)),
\end{split}
\]
if $v_c$ is odd, then $M'=M$. 
\end{remark}

\subsection{An example}

\begin{example}
Consider an isolated cDV singularity of type $cD_6$ given by
\[
   f(x,y,z,w)=x^2+y^2z+z^7+z^3w^2+z^2w^3+w^7.
\]
The Newton boundary $\Gamma(h)$ 
of $h(z,w)=z^7+z^3w^2+z^2w^3+w^7$ has four vertices ${\bf v}_0=(0,7)$, ${\bf v}_1=(2,3)$, ${\bf v}_2=(3,2)$, ${\bf v}_3=(7,0)$.
Since $a=0$ in the form~\eqref{eq1-1} and the power of the term $w^7$ is odd,
we have $c=0$ and $\varepsilon_1=\varepsilon_2=0$ in Theorem~\ref{thm2}.
In particular, all $\Delta_i$ belong to Case~(A) in Lemma~\ref{lemma43z}.
Since
$(u_1-u_0,v_0-v_1)=(2,4)$,
$(u_2-u_1,v_1-v_2)=(1,1)$, and
$(u_3-u_2,v_2-v_3)=(4,2)$,
the inequality $\nu_2(v_{i-1}-v_i)>\nu_2(\ell_i)$ holds only when $i=1$.
Therefore, $\text{\rm rank\,}H_2(K_f)=\ell_1=\gcd(2,4)=2$ by Theorem~\ref{thm2}.
 
The zeta function $\zeta_f(t)$ can be determined by applying the lemmas in this section.
Since $a=0$ in the form~\eqref{eq1-1} and the terms $z^7$ and $w^7$ exist, 
the faces of $\Gamma(f)$ belong to Cases~(a), (d), (e) and (f).
It follows from Lemma~\ref{lemma43a} that
the factors corresponding to the faces in Case~(a) are
\[
   \zeta_{f, \Delta_1}(t)=(1-t^7)^2,
   \quad 
   \zeta_{f, \Delta_2}(t)=\frac{1-t^{10}}{1-t^5}
   \quad \text{and}\quad 
   \zeta_{f, \Delta_3}(t)=\frac{(1-t^{14})^2}{(1-t^7)^2}.
\]
The factors corresponding to the faces in Case~(d) and in Case~(e) are
\[
\zeta_{f,{\bf v}_3\cap X_3}(t)=\frac{1-t^7}{1-t^{14}}
   \quad \text{and}\quad 
\zeta_{f,{\bf v}_0\cap X_4}(t)=\frac{1-t^{14}}{1-t^7}
\]
by Lemma~\ref{lemma44} and by Lemma~\ref{lemma43}, respectively.
Taking the product of these factors together with the factor $(1-t^2)^{-1}$ corresponding to the face in Case~(f), we obtain
\[
   \zeta_f(t)=\frac{(1-t^{10})(1-t^{14})^2}{(1-t^5)(1-t^2)}.
\]
The Milnor number $\mu_f$ of $(f,O)$ is
$\mu_f=\deg \det(I-tT)=\deg\zeta_f(t)+1=32$.
\end{example}

\section{Weighted homogeneous case}\label{sec5}

The torsion subgroup of $H_{k-2}(K_f)$ for weighted homogeneous polynomials is discussed in~\cite{Orl72}.
A singularity $(f,O)$ is said to be {\it weighted homogeneous} if $f$ is weighted homogeneous.
If $(f,O)$ is weighted homogeneous and isolated,
then there exists a positive integer $d$ such that $T^d-I=0$,
where $T$ is the monodromy matrix of the homological monodromy $H_{k-1}(F_f)\to H_{k-1}(F_f)$
associated with the Milnor fibration of $(f,O)$.
This property is used to study the torsion subgroup of $H_{k-2}(K_f)$ 
via a decomposition of the characteristic polynomial $\det(I-tT)$ into cyclotomic polynomials.
Therefore, we restrict our attention to the case where an isolated cDV singularity $(f,O)$ is weighted homogeneous.
In this case, the singularity is always Newton non-degenerate.
Furthermore, we assume that $(f,O)$ is a Thom-Sebastiani sum of Brieskorn-Pham, cyclic, or chain type singularities, since we need to apply Orlik's conjecture as explained in the introduction.

\begin{definition}\label{definition5-1}
A singularity $(f,O)$ is said to be of {\it Brieskorn-Pham}, {\it cyclic} and {\it chain type}
 if $f$ is of the form
\[
\begin{split}
   f(x_1,x_2,\ldots, x_k)&=x_1^{a_1}+x_2^{a_2}+\cdots +x_k^{a_k}, \\
   f(x_1,x_2,\ldots, x_k)&=x_1^{a_1}x_2+x_2^{a_2}x_3+\cdots +x_{k-1}^{a_{k-1}}x_k+x_k^{a_k}x_1, \\
   f(x_1,x_2,\ldots, x_k)&=x_1^{a_1}x_2+x_2^{a_2}x_3+\cdots +x_{k-1}^{a_{k-1}}x_k+x_k^{a_k},
\end{split}
\]
respectively, where $a_1,a_2,\ldots, a_k$ are integers with $a_i\geq 2$.
A singularity $(F,O)$ is called the {\it Thom-Sebastiani sum} of two singularities $(f,O)$ and $(g,O)$ if
$F$ is given as
\[
   F(x_1,\ldots,x_k,y_1,\ldots,y_{k'})=f(x_1,\ldots,x_k)+g(y_1,\ldots,y_{k'}),
\]
where $f:\C^k\to\C$, $g:\C^{k'}\to\C$, and $F:\C^{k+{k'}}\to\C$.
\end{definition}

Note that Brieskorn-Pham, cyclic, and chain type singularities,
and their Thom-Sebastiani sums are weighted homogeneous.

Let $(f,O)$ be a weighted homogeneous, isolated cDV singularity given in the form~\eqref{eq1-1}.
In particular, $(f,O)$ is Newton non-degenerate.
Therefore, we can assume that $h(z,w)$ is in the form $bz^pw^q+cz^rw^s$ for $b, c\in\C$.
If $a=0$, then neither $z$ nor $w^2$ divides $h(z,w)$ by Lemma~\ref{lemma21}~(1).
Therefore, $f(x,y,z,w)$ is in the form
\begin{equation}\label{eq51}
      f(x,y,z,w)=x^2+y^2z+z^pw^q+w^s,
\end{equation}
where $p+q\geq 4$, $q\in \{0,1\}$, and $s\geq 4$.
If $q=0$, then $(f,O)$ is the Thom-Sebastiani sum of the Brieskorn-Pham singularity $(x^2+w^s, O)$ and 
the chain type singularity $(y^2z+z^p,O)$.
If $q=1$, then $f$ is the Thom-Sebastiani sum of the Brieskorn-Pham singularity $(x^2, O)$ and 
the chain type singularity $(y^2z+z^pw+w^s,O)$.

If $a\ne 0$, then the condition for $(f,O)$ to be weighted homogeneous is given as follows.

\begin{lemma}\label{lemma52}
Let $(f,O)$ be a singularity given in the form
\begin{equation}\label{eq_lemma52}
      f(x,y,z,w)=x^2+y^2z+ayw^\kappa+bz^pw^q+cz^rw^s.
\end{equation}
Assume that $a,b,c\ne 0$. Then, $(f,O)$ is weighted homogeneous if and only if $2(p-r)\kappa=(p+1)s-q(r+1)$ holds.
\end{lemma}

\begin{proof}
From the condition that the five vertices $(2,0,0,0)$, $(0,2,1,0)$, $(0,1,0,\kappa)$, $(0,0,p,q)$, and $(0,0,r,s)$ lie on a hyperplane $\alpha_1X_1+\alpha_2X_2+\alpha_3X_3+\alpha_4X_4=\beta$, eliminating $\alpha_1$, $\alpha_2$ and $\alpha_3$ yields $(q+2p\kappa)(r+1)=(s+2r\kappa)(p+1)$, equivalently $2(p-r)\kappa=(p+1)s-q(r+1)$.
\end{proof}

Lemma~\ref{lemma52} implies that there are infinitely many weighted homogeneous singularities even when $a,b,c\ne 0$.
For example, one may set $p=r+1$ and choose $q$ and $s$ to be even.
A singularity $(f,O)$ given in the form~\eqref{eq_lemma52} with $a\ne 0$ is a Thom-Sebastiani sum of Brieskorn-Pham, cyclic, or chain type singularities only if either $b=0$ or $c=0$.
If $b=c=0$, then $(f,O)$ is non-isolated. 
Therefore, we may assume that $b\ne 0$ and $c=0$.
Furthermore, $q\in \{0,1\}$ by Lemma~\ref{lemma21}~(2).
In summary, $(f,O)$ is in the form
\begin{equation}\label{eq52}
      f(x,y,z,w)=x^2+ay^2z+yw^\kappa+bz^pw^q,
\end{equation}
where $\kappa\geq 3$, $p+q\geq 4$, $q\in\{0,1\}$, and $a,b\ne 0$.
If $q=0$, then $f$ is the Thom-Sebastiani sum of the Brieskorn-Pham singularity $(x^2, O)$ and 
the chain type singularity $(y^2z+ayw^\kappa+bz^p,O)$.
If $q=1$, then $f$ is the Thom-Sebastiani sum of the Brieskorn-Pham singularity $(x^2, O)$ and 
the cyclic type singularity $(y^2z+ayw^\kappa+bz^pw,O)$.

If $(f,O)$ is a Thom-Sebastiani sum of Brieskorn-Pham, cyclic, or chain type singularities,
then Orlik's conjecture holds by~\cite[Theorem 1.3 (a) and (c)]{HM22}.
Therefore, the torsion subgroup of $H_2(K_f)$ can be determined from the cyclotomic polynomials
appearing as factors of the characteristic polynomial of $f$.

\begin{theorem}\label{thm51}
Let $(f,O)$ be a weighted homogeneous, isolated cDV singularity of type $cD_n$ with $n>4$.
Assume that $(f,O)$ is a Thom-Sebastiani sum of Brieskorn-Pham, cyclic, or chain type singularities.
\begin{itemize}
\item[(1)] 
Suppose that $a=0$, where $(f,O)$ is given in the form~\eqref{eq51} with $p+q\geq 4$, $q\in \{0,1\}$, and $s\geq 4$. Then
\[
   H_2(K_f)\cong
   \begin{cases}
      \Z^{\ell+1} & \text{if $(s-q)/\ell$ and $s$ are even}\,;  \\ 
      \Z^{\ell} & \text{if $(s-q)/\ell$ is even and $s$ is odd}\,;  \\ 
      \Z\oplus \Z_2^{\ell+q-2} & \text{if $(s-q)/\ell$ is odd and  $s$ even}\,; \\
      \Z_2^{\ell+q-1} & \text{if $(s-q)/\ell$ and $s$ are odd,}
   \end{cases}
\]
where $\ell=\gcd(p, s-q)$.
\item[(2)]
Suppose that $a\ne 0$, where $(f,O)$ is given in the form~\eqref{eq52} with $\kappa\geq 3$, $p+q\geq 4$, $q\in \{0,1\}$, and $b\ne 0$.
Then.
\[
   H_2(K_f)\cong 
   \begin{cases}
   \Z_2^{M-1} & \text{if $q=1$}\,; \\
   \Z_2^{M-2} & \text{if $q=0$ and $M'=M$}\,; \\
   \Z^M & \text{if $q=0$ and $M'=2M$}, 
   \end{cases}
\]  
where $M=\gcd(2\kappa-q, p+1)$ and $M'=\gcd(2\kappa-q, 2(p+1))$.
\end{itemize}
\end{theorem}

\begin{remark}
Smale proved in~\cite{Sma62} that 
if a simply connected, closed $5$-dimensional manifold $K$ has vanishing second Stiefel-Whitney class,
then the exponent of each torsion factor of $H_2(K)$ is even.
Therefore, the exponent of the $\Z_2$-torsion of $H_2(K_f)$ is always even.
This can be verified directly from the statement of Theorem~\ref{thm51}.
For example, in the case of assertion~(1) with $(s-q)/\ell$ odd and $s$ even,
the oddness of $(s-q)/\ell$ implies that $\ell$ is odd when $q=1$, and even when $q=0$. 
Therefore, $\ell+q-2$ is even in either case.
\end{remark}

\begin{example}
Consider an isolated cDV singularity of type $cD_n$ given by
\[
   f(x,y,z,w)=x^2+y^2z+z^p+w^s\quad \text{with $s$ odd and $\gcd(p,s)=1$.}
\]
This is the case of assertion~(1) with $(s-q)/\ell$ and $s$ being odd, and $q=0$.
Since the exponent $\ell+q-1$ of $\Z_2$ is $0$, $H_2(K_f)$ is trivial.
Hence, by~\cite{Sma62}, $K_f$ is diffeomorphic to $S^5$.
For each $m>0$, an isolated cDV singularity $(f,O)$ of type $cD_n$ whose link $K_f$ is diffeomorphic to 
the connected sum of $m$ copies of $S^2\times S^3$ is obtained, for example, 
from the first and second cases in assertion~(1) of Theorem~\ref{thm51}. 
\end{example}

To prove Theorem~\ref{thm51}, we will use the following lemma.

\begin{lemma}[\cite{Orl72}, {\cite[Lemma 13.3]{HM22}}]
Let $f(x_1,x_2,\ldots,x_k)$ be a weighted homogeneous polynomial with $f(O)=0$.
Suppose that $(f,O)$ is an isolated singularity.
Let $\alpha_1(t), \alpha_2(t), \ldots, \alpha_\tau(t)$ be the elementary divisors of the characteristic polynomial $\det(I-tT)=(1-t)\zeta(t)$ of $(f,O)$, that is, 
they are products of cyclotomic polynomials with multiplicities $1$ and 
satisfying 
\[
   \det(I-tT)=\alpha_1(t)\alpha_2(t)\cdots \alpha_\tau(t)
\]
and
\[
   \alpha_\tau(t)\hd \alpha_{\tau-1}(t) \hd \cdots \hd \alpha_2(t) \hd \alpha_1(t).
\]
If Orlik's conjecture holds true for $(f,O)$, then
\[
   H_{k-2}(K_f)\cong \Z^m\oplus \bigoplus_{j=m+1}^\tau\Z_{\alpha_j(1)},
\]
where $m$ is the multiplicity of $(1-t)$ in $\det(I-tT)$.
\end{lemma}

The elementary divisors are obtained as follows.
Let $\det(I-tT)=\prod_{d\in J_1} (\Phi_d(t))^{m_1(d)}$ 
be the factorization of $\det(I-tT)$ with cyclotomic polynomials $\Phi_d(t)$,
where 
$\Phi_d(t)$ is the $d$-th cyclotomic polynomial, 
$m_1(d)\in\N$, 
and $J_1\subset \N$ is the set of indices $d$ such that $\Phi_d(t)$ appears in the factorization.
Then, set
\[
   \alpha_1(t)=\prod_{d\in J_1} \Phi_d(t).
\]
Next, consider the factorization of $\det(I-tT) (\alpha_1(t))^{-1}$,
which is given as $\det(I-tT) (\alpha_1(t))^{-1}=\prod_{d\in J_2} (\Phi_d(t))^{m_2(d)}$,
where
$m_2(d)\in\N$, and
$J_2\subset J_1$ is again the set of indices $d$ such that $\Phi_d(t)$ appears in this factorization.
Then, set
\[
   \alpha_2(t)=\prod_{d\in J_2} \Phi_d(t).
\]
We then consider  the factorization of $\det(I-tT) (\alpha_1(t))^{-1}(\alpha_2(t))^{-1}$ 
and set $\alpha_3(t)$ in the same manner.
We continue this process until no factor remains.
Thus,
we obtain $\det(I-tT) (\alpha_1(t))^{-1}(\alpha_2(t))^{-1}\cdots (\alpha_\tau(t))^{-1}=1$, and 
the first condition of the elementary divisors is satisfied. The second condition is obviously satisfied from the construction of the elementary divisors.

The value $\Phi_d(1)$ is 
\[
   \Phi_d(1)=\begin{cases} \alpha & \text{if $d=\alpha^\beta$ for some $\alpha, \beta\in\N$ with $\alpha$ prime,} \\
   1 & \text{otherwise.}
\end{cases}
\]
Therefore, only factors of the form $\Phi_{\alpha^\beta}(t)$ contribute to the torsion subgroup, 
giving a $\Z_\alpha$ summand. 
Note that, as stated in the above lemma, the polynomials $\alpha_i(t)$ for $i=1,\ldots,m$ do not contribute to the torsion subgroup.

For each matrix $I-tT$, there exist two matrices $U(t)$ and $V(t)$ such that $U(t)(I-tT)V(t)$ is a diagonal matrix with diagonal elements 
\[
  \begin{bmatrix} 
  \displaystyle 
  \prod_{i=1}^\tau \alpha_i(t), & 
  \displaystyle 
  \prod_{i=2}^\tau \alpha_i(t), & \cdots, & 
  \displaystyle 
  \prod_{i=\tau-1}^\tau \alpha_i(t), & \alpha_\tau(t), & 1, & \cdots, & 1 \end{bmatrix}.
\]
In the following proof, we write down this matrix using cyclotomic polynomials and determine the torsion subgroup for each case.

\begin{proof}[Proof of Theorem~\ref{thm51}]
We first prove assertion~(1).
Suppose that $a=0$.
First, we consider the case where $(s-q)/\ell$ is even.
If $q=1$, then $\Gamma(f)$ consists of the faces of the simplex drawn on the right in Figure~\ref{fig2}
with ${\bf v}_i=(p,1)$ and ${\bf v}_{i-1}=(0,s)$, together with the faces in Figure~\ref{fig4}
with ${\bf v}_0\cap X_4={\bf v}_{i-1}=(0,s)$.
Therefore, the characteristic polynomial $\det(I-tT)=(1-t)\zeta_f(t)$ of $(f,O)$ is obtained by applying
Lemma~\ref{lemma43a}, with $(v_{i-1}-v_i)/\ell_i=(s-q)/\ell$ even, and Lemma~\ref{lemma43}, with $v_0=s$, as
\begin{equation}\label{eq61}
   \det(I-tT)=
\begin{cases}
\displaystyle 
(1-t)(1-t^{ps/\ell})^{\ell}(1-t^s)\frac{1}{1-t^2} & \text{if $s$ is even,} 
\vspace{2mm}
\\
\displaystyle 
(1-t)(1-t^{ps/\ell})^{\ell}\frac{1-t^{2s}}{1-t^s}\frac{1}{1-t^2} & \text{if $s$ is odd},
\end{cases}
\end{equation}
where the factor $(1-t^2)^{-1}$ corresponds to the vertex ${\bf u}_1$ as in Section~\ref{sec4-8}. 
If $s$ is even, then the multiplicity of $\Phi_1(t)=1-t$ is $\ell+1$ and that of $\Phi_d(t)$ is at most $\ell+1$ for any $d>1$.
Therefore, $\text{\rm rank}\,H_2(K_f)=\ell+1$ and $H_2(K_f)$ has no torsion subgroup.
If $s$ is odd, then the multiplicity of $\Phi_1(t)$ is $\ell$.
The right-hand side of~\eqref{eq61} is expressed in terms of cyclotomic polynomials as
\begin{equation}\label{eqc-99}
\Big(
\prod_{\substack{d_1>1 \\ d_1\hd (ps/\ell) }}(\Phi_{d_1}(t))^\ell
\Big)
\,
\Big(
\prod_{\substack{d_2>1 \\ d_2\hd 2s\\ d_2\not\hspace{0.7mm}|\hspace{0.7mm} s}}\Phi_{d_2}(t)
\Big)
\,
\Big(\Phi_2(t)\Big)^{-1},
\end{equation}
where the last factor of~\eqref{eqc-99} is obtained from $(1-t)(1-t^2)^{-1}$.
Since $s$ is odd, the middle factor of~\eqref{eqc-99} contains $\Phi_2(t)$, and this cancels 
the last factor $(\Phi_2(t))^{-1}$ of~\eqref{eqc-99}.
The remaining $d_2$'s are even and not powers of two, and therefore satisfy $\Phi_{d_2}(1)=1$. 
Hence, the multiplicity of $\Phi_{d_1}(t)$ with $\Phi_{d_1}(1)>1$ is at most $\ell$ for any $d_1>1$,
which implies that $\text{\rm rank\,}H_2(K_f)=\ell$ and $H_2(K_f)$ has no torsion subgroup.

If $q=0$ then $\Gamma(f)$ consists of the faces of the simplex drawn on the right in Figure~\ref{fig2}
with ${\bf v}_i=(p,0)$ and ${\bf v}_{i-1}=(0,s)$, together with the faces in Figure~\ref{fig4}
with ${\bf v}_0\cap X_4={\bf v}_{i-1}=(0,s)$
and those in Figure~\ref{fig3} with ${\bf v}_N\cap X_3={\bf v}_i=(p,0)$.
By Lemma~\ref{lemma44} with $u_N=p$, 
it follows that $\det(I-tT)$ is obtained from the right-hand side of~\eqref{eq61} by multiplying 
the factor 
\begin{equation}\label{eq62}
   \frac{1-t^p}{1-t^{2p}}.
\end{equation}
The rest of the argument is exactly the same and 
we conclude that the rank of $H_2(K_f)$ is $\ell+1$ if $s$ is even and $\ell$ if $s$ is odd, 
and $H_2(K_f)$ has no torsion subgroup. 
This proves the first two cases in assertion~(1).

Next, we consider the case where $(s-q)/\ell$ is odd and $s$ is even.
Suppose that $q=1$.
Then, $\det(I-tT)$ is obtained by applying Lemma~\ref{lemma43a}, with $(v_{i-1}-v_i)/\ell_i=(s-q)/\ell$ odd, 
and Lemma~\ref{lemma43}, with $v_0=s$ even, as
\begin{equation}\label{eq63}
   \det(I-tT)
=(1-t)\frac{(1-t^{2ps/\ell})^\ell}{(1-t^{ps/\ell})^{\ell}}(1-t^s)\frac{1}{1-t^2}.
\end{equation}
Since the multiplicity of the factor $\Phi_1(t)$ is $1$, we have $\text{\rm rank\,}H_2(K_f)=1$.

The right-hand side of~\eqref{eq63} is expressed in terms of cyclotomic polynomials as
\begin{equation}\label{eqc-1}
\Phi_1(t)
\;
\Big(
\prod_{\substack{d_1>1 \\ d_1\hd (2ps/\ell) \\ d_1\not\hspace{0.7mm}|\hspace{0.7mm}(ps/\ell)}}(\Phi_{d_1}(t))^\ell
\Big)
\,
\Big(
\prod_{\substack{d_2>1 \\ d_2\hd s\\ d_2\not\hspace{0.7mm}|\hspace{0.7mm} 2}}\Phi_{d_2}(t)
\Big),
\end{equation}
where the last factor of~\eqref{eqc-1} is obtained from $(1-t^s)(1-t^2)^{-1}$.
Since $\ell=\gcd(p-r, s-q)=\gcd(p,s-1)$, it follows that $\ell\hd p$.
Recall that $\nu_2(\alpha)$ denotes the $2$-adic valuation of an integer $\alpha$, that is, the integer $\beta$ such that 
$\alpha=2^\beta\gamma$ with $\gamma$ odd.
The $d_1$'s in the middle factor of~\eqref{eqc-1} satisfy
\[
   \nu_2(d_1)=\nu_2(2ps/\ell)
   =\nu_2(p/\ell)+\nu_2(s)+1,
\]
while the $d_2$'s in the last factor of~\eqref{eqc-1} satisfy $\nu_2(d_2)\leq \nu_2(s)$.
Hence $\nu_2(d_1)>\nu_2(d_2)$, and therefore there are no $d_1$ and $d_2$ with $d_1=d_2$.
This implies that 
the multiplicity of each $\Phi_{d_1}(t)$ in $\det(I-tT)$ is $\ell$ and that of each $\Phi_{d_2}(t)$ is $1$.
Thus, the diagonal elements of $U(t)(I-tT)V(t)$ are
\begin{equation}\label{eqd-1}
\begin{bmatrix}
\Phi_1(t)\Big(\prod \Phi_{d_1}(t)\Big)\Big(\prod \Phi_{d_2}(t)\Big), & 
\prod \Phi_{d_1}(t), & \cdots, &
\prod \Phi_{d_1}(t), & 1, & \cdots, & 1
\end{bmatrix},
\end{equation}
where the factor $\prod \Phi_{d_1}(t)$ appears in the first $\ell$ entries.
There is a unique power of two $d_1$, which satisfies $\Phi_{d_1}(1)=2$.
The other $d_1$'s are not prime since they are even, and hence satisfy $\Phi_{d_1}(1)=1$.
Therefore, substituting $t=1$ into~\eqref{eqd-1}, all entries from the second to the $\ell$-th are equal to $2$.
The first entry does not contribute since $\text{\rm rank\,}H_2(K_f)=1$,
and the $2$'s from the second to the $\ell$-th entries give rise to the torsion subgroup $\Z_2^{\ell-1}$.
Therefore, we have $H_2(K_f)\cong \Z\oplus\Z_2^{\ell-1}$.

Suppose that $q=0$. 
By applying Lemma~\ref{lemma44} with $u_N=p$, it follows that 
$\det(I-tT)$ is obtained from the right-hand side of~\eqref{eq63} by multiplying 
the factor~\eqref{eq62} as
\begin{equation}\label{eq67}
\det(I-tT)
=(1-t)\frac{(1-t^{2ps/\ell})^\ell}{(1-t^{ps/\ell})^{\ell}}(1-t^s)\frac{1}{1-t^2}\frac{1-t^p}{1-t^{2p}}.
\end{equation}
Since the multiplicity of the factor $\Phi_1(t)$ is $1$, we have $\text{\rm rank\,}H_2(K_f)=1$.

The right-hand side of~\eqref{eq67} is expressed as
\begin{equation}\label{eqc-2}
\Phi_1(t)
\;
\Big(
\prod_{\substack{d_1>1 \\ d_1\hd (2ps/\ell) \\ d_1\not\hspace{0.7mm}|\hspace{0.7mm}(ps/\ell)}}(\Phi_{d_1}(t))^\ell
\Big)
\,
\Big(
\prod_{\substack{d_2>1 \\ d_2\hd s\\ d_2\not\hspace{0.7mm}|\hspace{0.7mm} 2}}\Phi_{d_2}(t)
\Big)
\,
\Big(
\prod_{\substack{d_3>1 \\ d_3\hd 2p\\ d_3\not\hspace{0.7mm}|\hspace{0.7mm} p}}\Phi_{d_3}(t)\Big)^{-1}.
\end{equation}
By the same argument as in the case $q=1$, we can conclude that there are no $d_1$ and $d_2$ with $d_1=d_2$.
Therefore, the multiplicity of each $\Phi_{d_2}(t)$ in $\det(I-tT)$ is at most $1$.

Consider the multiplicity of $\Phi_{d_1}(t)$.
Since $\ell=\gcd(p,s)$, it follows that $\ell \hd s$.
Then, the condition $(s-q)/\ell=s/\ell$ being odd implies that $\nu_2(s/\ell)=0$.
The $d_1$'s and $d_3$'s in~\eqref{eqc-2} satisfy $\nu_2(d_1)=\nu_2(2ps/\ell)=\nu_2(2p)$ and $\nu_2(d_3)=\nu_2(2p)$,
respectively.
Hence, there are $d_1$ and $d_3$ with $d_1=d_3$.
The diagonal elements of $U(t)(I-tT)V(t)$ are
\begin{equation}\label{eqd-2}
\begin{split}
\Big[\displaystyle 
\Phi_1(t)
\Big(\prod \Phi_{d_1}(t)\Big)
&
\Big(\prod \Phi_{d_2}(t)\Big) 
\Big(\prod \Phi_{d_3}(t)\Big)^{-1}, \;\, 
\prod \Phi_{d_1}(t),
\;\, \\
& \cdots, \;\, 
\prod \Phi_{d_1}(t),
\;\,
\prod_{d_1\not\in\{d_3\}} \Phi_{d_1}(t), \;\, 1, \;\, \cdots, \;\, 1
\Big],
\end{split}
\end{equation}
where 
the factor $\prod_{d_1\not\in\{d_3\}} \Phi_{d_1}(t)$ appears in the $\ell$-th entry.
There is a unique power of two $d_1$, which satisfies $\Phi_{d_1}(1)=2$.
The other $d_1$'s satisfy $\Phi_{d_1}(1)=1$ since $d_1$ is even.
Since any $d_1$ with $\Phi_{d_1}(1)>1$ is a power of two and the same power of two appears in the $d_3$'s, there is no factor in the $\ell$-th entry $\prod_{d_1\not\in\{d_3\}} \Phi_{d_1}(t)$ of~\eqref{eqd-2} with $\Phi_{d_1}(1)=2$.
Thus, we have $H_2(K_f)\cong \Z\oplus \Z_2^{\ell-2}$.
This proves the third case in assertion~(1).

Finally, consider the case where $(s-q)/\ell$ and $s$ are odd.
Suppose that $q=1$.
Then, $\det(I-tT)$ is obtained by applying Lemma~\ref{lemma43a}, with $(v_{i-1}-v_i)/\ell_i=(s-q)/\ell$ odd, 
and Lemma~\ref{lemma43}, with $v_0=s$ odd, as
\begin{equation}\label{eq64}
\det(I-tT)
=(1-t)\frac{(1-t^{2ps/\ell})^\ell}{(1-t^{ps/\ell})^{\ell}}\frac{1-t^{2s}}{1-t^s}\frac{1}{1-t^2}.
\end{equation}
Since the multiplicity of the factor $\Phi_1(t)$ is $0$, $H_2(K_f)$ has no free part.

The right-hand side of~\eqref{eq64} is expressed as
\begin{equation}\label{eqc-3}
\Big(
\prod_{\substack{d_1>1 \\ d_1\hd (2ps/\ell) \\ d_1\not\hspace{0.7mm}|\hspace{0.7mm}(ps/\ell)}}(\Phi_{d_1}(t))^\ell
\Big)\,
\Big(
\prod_{\substack{d_2>1 \\ d_2\hd 2s\\ d_2\not\hspace{0.7mm}|\hspace{0.7mm} s}}\Phi_{d_2}(t)
\Big)
\,
\Big(\Phi_2(t)\Big)^{-1},
\end{equation}
where the last factor of~\eqref{eqc-3} is obtained from $(1-t)(1-t^2)^{-1}$.
Since $s$ is odd,  the middle factor of~\eqref{eqc-3} contains $\Phi_2(t)$, and this cancels 
the last factor $(\Phi_2(t))^{-1}$ of~\eqref{eqc-3}.
The remaining $d_2$'s satisfy $\nu_2(d_2)=1$. 
Since these $d_2$'s are even and not powers of two, if there exist $d_1$ and $d_2$ with $d_1=d_2$, then they satisfy $\Phi_{d_1}(1)=\Phi_{d_2}(1)=1$.
Therefore, they will not contribute to the calculation of the torsion subgroup of $H_2(K_f)$.
To avoid cumbersome notation, we present the diagonal elements of $U(t)(I-tT)V(t)$
while omitting the factors $\Phi_d(t)$ with $\Phi_d(1)=1$. This omission does not affect the calculation of the torsion subgroup.
With these factors omitted, the diagonal elements are expressed as
\begin{equation}\label{eqd-3}
\begin{split}
\Big[\displaystyle 
\Big(\prod\nolimits^* \Phi_{d_1}(t)\Big)
\Big(\prod\nolimits^* \Phi_{d_2}(t)\Big) 
&
\Big(\Phi_2(t)\Big)^{-1}, \;\, 
\prod\nolimits^* \Phi_{d_1}(t),
\;\, \\
& \cdots, \;\, 
\prod\nolimits^* \Phi_{d_1}(t),
\;\,
\prod\nolimits^*_{d_1\not\in\{d_2\}\setminus\{2\}} \Phi_{d_1}(t), \;\, 1, \;\, \cdots, \;\, 1
\Big],
\end{split}
\end{equation}
where the factor $\prod\nolimits^*_{d_1\not\in\{d_2\}\setminus\{2\}} \Phi_{d_1}(t)$ appears in the $\ell$-th entry.
The symbol $\prod\nolimits^*$ denotes the product of $\Phi_d(t)$'s satisfying $\Phi_d(1)>1$,
and is defined to be $1$ if there are no such factors.
Since $d_1=d_2$ implies $\Phi_{d_1}(1)=\Phi_{d_2}(1)=1$ unless $d_2=2$, all factors in $\prod\nolimits^* \Phi_{d_1}(t)$ with $\Phi_{d_1}(1)=2$ remain in the $\ell$-th entry $\prod\nolimits^*_{d_1\not\in\{d_2\}\setminus\{2\}} \Phi_{d_1}(t)$.
Thus, substituting $t=1$ into~\eqref{eqd-3}, the first $\ell$ entries are equal to $2$ and the remaining entries are $1$, which implies $H_2(K_f)\cong \Z_2^\ell$.

Suppose that $q=0$.
By applying Lemma~\ref{lemma44} with $u_N=p$, it follows that $\det(I-tT)$ is obtained from the right-hand side of~\eqref{eq64} by multiplying the factor~\eqref{eq62} as
\begin{equation}\label{eq68}
\det(I-tT)
=(1-t)\frac{(1-t^{2ps/\ell})^\ell}{(1-t^{ps/\ell})^{\ell}}\frac{1-t^{2s}}{1-t^s}\frac{1}{1-t^2}\frac{1-t^p}{1-t^{2p}}.
\end{equation}
Since the multiplicity of the factor $\Phi_1(t)$ is $0$, $H_2(K_f)$ has no free part.

The right-hand side of~\eqref{eq68} is expressed as
\[
\Big(
\prod_{\substack{d_1>1 \\ d_1\hd (2ps/\ell) \\ d_1\not\hspace{0.7mm}|\hspace{0.7mm}(ps/\ell)}}(\Phi_{d_1}(t))^\ell
\Big)
\,
\Big(
\prod_{\substack{d_2>1 \\ d_2\hd 2s\\ d_2\not\hspace{0.7mm}|\hspace{0.7mm} s}}\Phi_{d_2}(t)
\Big)
\,
\Big(\Phi_2(t)\Big)^{-1}
\,
\Big(\prod_{\substack{d_3>1 \\ d_3\hd 2p\\ d_3\not\hspace{0.7mm}|\hspace{0.7mm} p}}\Phi_{d_3}(t)\Big)^{-1}.
\]
As in the case of $q=1$, the diagonal elements of $U(t)(I-tT)V(t)$, after omitting the factors $\Phi_d(t)$ with $\Phi_d(1)=1$, are expressed as
\begin{equation}\label{eqd-4}
\begin{split}
\Big[\displaystyle 
\Big(\prod\nolimits^* \Phi_{d_1}(t)\Big)
\Big(\prod\nolimits^* \Phi_{d_2}(t)\Big) 
&
\Big(\Phi_2(t)\Big)^{-1}
\Big(\prod\nolimits^* \Phi_{d_3}(t)\Big)^{-1}, \;\, 
\prod\nolimits^* \Phi_{d_1}(t),
\;\, \\
& \cdots, \;\, 
\prod\nolimits^* \Phi_{d_1}(t),
\;\,
\prod\nolimits^*_{d_1\not\in\{d_3\}} \Phi_{d_1}(t), \;\, 1, \;\, \cdots, \;\, 1
\Big],
\end{split}
\end{equation}
where the factor $\prod\nolimits^*_{d_1\not\in\{d_3\}} \Phi_{d_1}(t)$ appears in the $\ell$-th entry.
In the above presentation, the information of $d_2$ is omitted, since it does not affect the calculation of the torsion subgroup, as shown in~\eqref{eqd-3} and the subsequent discussion.
Since $s$ is odd, $\nu_2(s)=\nu_2(\ell)=0$. Therefore, if $d_1$ and $d_2$ are prime powers, then $d_1=d_2$.
This implies that the factor in the $\ell$-th entry is $\prod\nolimits^*_{d_1\not\in\{d_3\}} \Phi_{d_1}(1)=1$. 
Therefore, substituting $t=1$ into~\eqref{eqd-4}, the first $\ell-1$ entries are equal to $2$ and the remaining entries are $1$.
Hence $H_2(K_f)\cong \Z_2^{\ell-1}$.
This completes the proof of assertion~(1).

Next, we prove assertion~(2), that is, $(f,O)$ is given in the form~\eqref{eq52} for $\kappa\geq 3$, $p+q\geq 4$, $q\in \{0,1\}$, and $a, b\ne 0$.
If $q=1$, then $\Gamma(f)$ consists of the faces of the simplex in Figure~\ref{fig7}
with ${\bf v}_c=(p,1)$. 
We apply Lemma~\ref{lemma43b} with $(u_c, v_c)=(p,1)$.
Since $v_c$ is odd, we have $M'=M$ by Remark~\ref{rem410}.
Therefore,
\begin{equation}\label{eq65}
   \det(I-tT)=(1-t)\frac{(1-t^{2(2\kappa p+1)/M})^M}{(1-t^{(2\kappa p+1)/M})^M}\frac{1}{1-t^2}.
\end{equation}
Since the multiplicity of the factor $\Phi_1(t)$ is $0$, $H_2(K_f)$ has no free part.

The right-hand side is expressed as
\begin{equation}\label{eqc-5}
\Big(
\prod_{\substack{d_1>1 \\ d_1\hd (2(2\kappa p+1)/M) \\ d_1\not\hspace{0.7mm}|\hspace{0.7mm}((2\kappa p+1)/M)}}(\Phi_{d_1}(t))^M
\Big)
\,
\Big(\Phi_2(t)\Big)^{-1},
\end{equation}
where the second factor of~\eqref{eqc-5} is obtained from $(1-t)(1-t^2)^{-1}$.
Since $2\kappa p+1$ is odd, the first factor of~\eqref{eqc-5} contains $\Phi_2(t)$,
and this cancels $(\Phi_2(t))^{-1}$.
Thus, the diagonal elements of $U(t)(I-tT)V(t)$ are
\begin{equation}\label{eqd-5}
\Big[
\Big(\prod \Phi_{d_1}(t)\Big)\Big(\Phi_{2}(t)\Big)^{-1}, \;\,
\prod \Phi_{d_1}(t), \;\, \cdots, \;\,
\prod \Phi_{d_1}(t), \;\,
{\displaystyle \prod_{d_1\ne 2}}  \Phi_{d_1}(t), \;\, 1, \;\, \cdots, \;\, 1
\Big],
\end{equation}
where the factor $\prod_{d_1\ne 2} \Phi_{d_1}(t)$ appears in the $M$-th entry.
Since $2\kappa p+1$ is odd, $\Phi_{d_1}(1)=1$ for any $d_1\ne 2$.
Therefore, substituting $t=1$ into~\eqref{eqd-5}, the first $M-1$ entries are equal to $2$ and the remaining entries are $1$.
Hence $H_2(K_f)\cong \Z_2^{M-1}$.
This proves the first case in assertion~(2).

Suppose that $q=0$. 
Then, $\Gamma(f)$ consists of the faces of the simplex in Figure~\ref{fig7}
with ${\bf v}_c=(p,0)$ and those in Figure~\ref{fig3} with ${\bf v}_N\cap X_3={\bf v}_c=(p,0)$. 
Applying Lemma~\ref{lemma43b} with $(u_c, v_c)=(p,0)$
and Lemma~\ref{lemma44} with $(u_N, v_N)=(p,0)$,
we obtain 
\begin{equation}\label{eq66}
   \det(I-tT)=
   \begin{cases} 
   \displaystyle 
   (1-t)\frac{(1-t^{4\kappa p/M})^M}{(1-t^{2\kappa p/M})^M}\frac{1}{1-t^2}\frac{1-t^p}{1-t^{2p}}
   & \text{if $M'=M$,} 
   \vspace{2mm}
   \\
   \displaystyle 
   (1-t)(1-t^{2\kappa p/M})^M\frac{1}{1-t^2}\frac{1-t^p}{1-t^{2p}}
   & \text{if $M'=2M$}.
   \end{cases}
\end{equation}
If $M'=2M$, then, since the multiplicity of the factor $\Phi_1(t)$ is $M$
and that of $\Phi_d(t)$ is at most $M$ for any $d>1$,
it follows that $\text{\rm rank}\,H_2(K_f)=M$ and $H_2(K_f)$ has no torsion subgroup.
This proves the third case in assertion~(2).

Suppose that $M=M'$.
Since the multiplicity of the factor $\Phi_1(t)$ is $0$, $H_2(K_f)$ has no free part.
The right-hand side of~\eqref{eq66} in the case $M'=M$ is expressed as
\begin{equation}\label{eqc-6}
\Big(
\prod_{\substack{d_1>1 \\ d_1\hd 4\kappa p/M \\ d_1\not\hspace{0.7mm}|\hspace{0.7mm}2\kappa p/M}}(\Phi_{d_1}(t))^M
\Big)
\,
\Big(\Phi_2(t)\Big)^{-1}
\,
\Big(\prod_{\substack{d_2>1 \\ d_2\hd 2p\\ d_2\not\hspace{0.7mm}|\hspace{0.7mm} p}}\Phi_{d_2}(t)\Big)^{-1}.
\end{equation}
Since $M=\gcd(2\kappa, p+1)$ and $M'=\gcd(2\kappa, 2(p+1))$, it follows from $M=M'$ that $M\geq 2$ and $\nu_2(2\kappa)\leq \nu_2(p+1)$.
In particular, $p$ is odd.
The oddness of $p$ implies that the last factor of~\eqref{eqc-6} contains $(\Phi_2(t))^{-1}$.
The first factor of~\eqref{eqc-6} contains $(\Phi_2(t))^M$.
Since $M\geq 2$, the factors $(\Phi_2(t))^{-1}$ in the second and third factors of~\eqref{eqc-6} cancel two of the $M$ factors $(\Phi_2(t))^M$ in the first factor of~\eqref{eqc-6}.
Therefore, the diagonal elements of $U(t)(I-tT)V(t)$ are expressed as
\begin{equation}\label{eqd-6}
\begin{split}
\Big[\displaystyle 
\Big(\prod \Phi_{d_1}(t)\Big)
&
\Big(\Phi_2(t)\Big)^{-1}
\Big(\prod \Phi_{d_2}(t)\Big)^{-1}, \;\, 
\prod \Phi_{d_1}(t),
\;\, \\
& \cdots, \;\, 
\prod \Phi_{d_1}(t),
\;\,
\prod_{d_1\ne 2} \Phi_{d_1}(t), 
\;\,
\prod_{d_1\not\in\{d_2\}} \Phi_{d_1}(t), 
\;\, 1, \;\, \cdots, \;\, 1
\Big],
\end{split}
\end{equation}
where the factor $\prod_{d_1\not\in\{d_2\}} \Phi_{d_1}(t)$ appears in the $M$-th entry.
The inequality $\nu_2(p+1)\geq \nu_2(2\kappa)$ implies that $\nu_2(2\kappa)=\nu_2(M)$, and hence $\nu_2(2\kappa p/M)=\nu_2(p)$. 
Since $p$ is odd, $2\kappa p/M$ is also odd. 
Hence, $\Phi_{d_1}(1)=1$ for any $d_1\ne 2$.
This implies that both the $(M-1)$-st entry and the $M$-th entry of~\eqref{eqd-6} are equal to $1$. 
Therefore, substituting $t=1$ into~\eqref{eqd-6}, the first $M-2$ entries are equal to $2$ and the remaining entries are $1$.
Hence $H_2(K_f)\cong \Z_2^{M-2}$.
This proves the second case in assertion~(2).
\end{proof}

\makeatletter
\@setaddresses
\makeatother

\end{document}